\documentclass{amsart}

\usepackage{layout} 
\usepackage[top=3cm, bottom=3cm, left=2cm, right=2cm]{geometry} 
\usepackage[english]{babel}
\usepackage[utf8]{inputenc}
\usepackage[T1]{fontenc}
\usepackage{amsmath}
\usepackage{amssymb}
\usepackage{mathrsfs}
\usepackage{amsthm}
\usepackage{enumerate} 
\usepackage[colorlinks=true,urlcolor=blue, citecolor=red,linkcolor=blue]{hyperref}
\usepackage{enumitem}
\usepackage{comment}
\usepackage{mathtools}
\usepackage{esint}
\usepackage{faktor}
\usepackage{color}
\usepackage[capitalize]{cleveref}
\newcommand{\doi}[1]{\href{https://doi.org/#1}{\nolinkurl{doi:#1}}}

\newtheorem{theorem}{Theorem}[section]
\newtheorem{proposition}[theorem]{Proposition}
\newtheorem{corollary}[theorem]{Corollary}
\newtheorem{lemma}[theorem]{Lemma}
\newtheorem{remark}[theorem]{Remark}

\newtheorem{claim}[theorem]{Claim}

\numberwithin{equation}{section}

\newcommand{\scal}[2]{\left\langle #1,#2 \right\rangle}

\newcommand{\dr}{\partial}
\newcommand{\bz}{\bar{z}}

\newcommand{\Span}{\mathrm{Span}}

\newcommand{\dd}{\mathrm{d}}

\newcommand{\vol}{\mathrm{vol}}

\newcommand{\tr}{\mathrm{tr}}
\newcommand{\Id}{\mathrm{Id}}

\newcommand{\W}{\mathrm{W}}
\newcommand{\Ll}{\mathrm{L}}
\newcommand{\Cc}{\mathrm{C}}
\newcommand{\G}{\mathrm{G}}

\newcommand{\VR}{\mathrm{V_R}}

\newcommand{\IL}{\mathrm{I^L}}

\newcommand{\R}{\mathbb{R}}
\newcommand{\C}{\mathbb{C}}
\newcommand{\N}{\mathbb{N}}

\newcommand{\s}{\mathbb{S}}
\newcommand{\T}{\mathbb{T}}
\newcommand{\Z}{\mathbb{Z}}
\newcommand{\D}{\mathbb{D}}
\newcommand{\Hd}{\mathbb{H}}

\newcommand{\Ar}{\mathcal{A}}

\newcommand{\Sr}{\mathcal{S}}

\newcommand{\Er}{\mathcal{E}}

\newcommand{\Pcal}{\mathcal{P}}

\newcommand{\Nc}{\mathscr{N}}
\newcommand{\Sc}{\mathscr{S}}

\newcommand{\II}{\mathrm{I\!I}}

\newcommand{\Ep}{\mathrm{Ep}}

\newcommand{\ust}{\underset}

\newcommand{\vp}{\varphi}
\newcommand{\eps}{\varepsilon}
\newcommand{\Chat}{\widehat{\C}}
\newcommand{\EOS}{{\rm E_{OS}}}

\newcommand{\ii}{\mathfrak{i}}

\title[Renormalized Volume and Area]{Comparison between renormalized volume and area in the Weil--Petersson topology}
\author{Dorian Martino}
\date{}

\usepackage[colorinlistoftodos,prependcaption,textsize=tiny]{todonotes}

\usepackage{parskip}

\begin{document}
	
	\maketitle
	
	\begin{abstract}
		We prove that the renormalized volume between the two Epstein surfaces associated with a Weil--Petersson quasicircle is continuous with respect to the Weil--Petersson topology. As a consequence, we establish that the universal Liouville action is proportional to this renormalized volume for every Weil--Petersson quasicircle. We then establish an inequality relating this renormalized volume to the renormalized area introduced by Alexakis--Mazzeo. To this end, we compare both quantities with the Hilbert--Schmidt norm of the associated Grunsky operator, which can be expressed in terms of a nonlocal Lagrangian introduced by O'Hara--Solanes.
	\end{abstract}

	
	\section{Introduction}
	
	Conformally invariant Lagrangians for submanifolds of Euclidean spaces can often be obtained by renormalizing an infinite volume. For instance, Graham--Witten \cite{GrahamWitten1999} introduced the notion of renormalized area or volume in 1999 for $(n+1)$-dimensional minimal submanifolds $Y^{n+1}$ of hyperbolic space $\Hd^{d+1} = \{(x,r) : x\in\R^d, r>0\}$ (or, more generally, of any Poincaré--Einstein manifold), in the context of the AdS/CFT correspondence (it is used, for instance, to compute the expected value of Wilson lines or the entanglement entropy, see \cite{GK2014,Anastasiou2025,Bueno2026}). They proved the following asymptotic expansion:
	\begin{equation*}
		\vol_{\Hd^{d+1}}\left(Y^{n+1}\cap \{r>\eps\}\right) \ust{\eps\to 0^+}{=} \frac{a_n}{\eps^n} + \cdots + \frac{a_1}{\eps} + b\, \log\left(\frac{1}{\eps}\right) + c + o(1).
	\end{equation*}
	When $n$ is even, Graham--Witten proved that $b$ defines a nontrivial conformally invariant Lagrangian on the boundary $\dr_{\infty}Y^{n+1}\subset \R^d$. When $n=2$, Graham--Witten recovered the Willmore energy, which is given by the square of the $\Ll^2$ norm of the mean curvature vector. More generally, when $n\geq 4$ is even, the coefficient $b$ can be described as a generalized Willmore energy \cite{Graham2017,GW2017} on $n$-dimensional submanifolds of $\R^d$ and is independent of the choice of the minimal extension in $\Hd^{d+1}$. One can then use these Lagrangians to study the set of immersions of a given abstract manifold into $\R^d$ up to conformal transformations, much as the length functional is used as a Morse-type functional to study the set of loops \cite{K1978}. This is, for instance, one of the motivations for studying the Willmore conjecture \cite{MN2014} or the $16\pi$ conjecture \cite{FSKHC1997,R2021}. We refer to the recent survey \cite{LMR2026} for further references.
	
	When $n$ is odd, $b=0$, and the coefficient $c$ (called the renormalized volume or area) depends on the chosen submanifold $Y$. For $n=1$, the computation was carried out by Alexakis--Mazzeo \cite{AlexakisMazzeo2010}, who obtained 
	\begin{equation}\label{eq:AM-functional}
		\Ar(Y) \coloneqq -2\pi\, \chi(Y) - \frac{1}{2}\int_Y |\II|^2_{g_Y}\, \dd\vol_{g_Y},
	\end{equation}
	where $\II$ is the second fundamental form of $Y$ in $(\Hd^3,g_{\Hd^3})$ and $g_Y$ denotes the first fundamental form. In this setting, any Jordan curve spans a properly embedded minimal disk \cite{S2004,S2005}.
	However, given a smooth Jordan curve $\Gamma\subset \R^2$, solutions to the asymptotic Plateau problem are generally not unique \cite{C2014}, which means that one cannot regard $c$ as a functional of the boundary submanifold. In 2008, Krasnov--Schlenker \cite{KrasnovSchlenker2008} introduced another way of computing the renormalized volume of hyperbolic 3-manifolds using the notion of Epstein surfaces described below. Similar constructions can be found in \cite{TT2003,K2000}. This approach was later developed by Bridgeman--Bromberg in \cite{BridgemanBromberg2026} to study relations between renormalized volume, W-volume, and the Osgood--Stowe differential.
	
	Epstein surfaces were introduced in 1984 \cite{Epstein1984} as a way of associating a surface in hyperbolic space with a conformal metric on a given domain of the sphere at infinity. In the setting considered here, they can be defined as follows.
	We consider a domain $\Omega\subset\Chat$ endowed with its complete Poincaré metric $h = e^{2u}|\dd z|^2$ of curvature $-1$. For each $z\in\Omega$, we consider the horosphere $H_z$ in $\Hd^3$ that is tangent to $\dr_{\infty}\Hd^3 = \C$ and has Euclidean radius $e^{-u(z)}$:
	\begin{equation*}
		H_z \coloneqq \left\{ (w,t)\in\C\times [0,+\infty) : |w-z|^2 + (t-e^{-u(z)})^2 = e^{-2u(z)} \right\}.
	\end{equation*}
	The Epstein surface associated with $\Omega$ is the envelope of these spheres in $\Hd^3$. An explicit expression is provided in \Cref{sec:epstein}. These envelopes can have singularities as surfaces, i.e., they may not be immersed everywhere, but they always admit a parametrization that is analytic in the interior of $\Omega$. A one-dimensional analog of these Epstein surfaces has recently been studied in \cite{VargasPalleteWangWolfram2025}, inspired by the holographic duality between Jackiw--Teitelboim gravity and Schwarzian field theory.
	
	A smooth Jordan curve $\Gamma\subset\Chat=\C\cup\{\infty\}$ splits $\Chat$ into two disjoint simply connected domains $\Omega_+$ and $\Omega_-$, each of which induces an Epstein surface in $\Hd^3$. Using the formulas of \cite{KrasnovSchlenker2008}, Bridgeman--Bromberg--Vargas Pallete--Wang \cite{BBVPW2025} defined the signed volume $V(\Gamma)$ between these surfaces and the renormalized volume
	\begin{equation}\label{eq:intro-renormalized-volume}
		\VR(\Gamma)=V(\Gamma) -\frac{1}{2}\int_{\Sigma_+\cup\Sigma_-}H\,\dd a,
	\end{equation}
	where $H=\frac{1}{2} \tr B$ is the scalar mean curvature, $B$ is the shape operator of the Epstein surfaces, and $\dd a$ is the signed area form. By differentiating $\VR$, they proved that it is, up to a factor of $4$, equal to the universal Liouville action $\Sr$ introduced by Takhtajan--Teo \cite{TakhtajanTeo2006} and defined as follows. Given two uniformizing maps $f\colon \D\to \Omega_+$ and $g\colon \D^*\coloneqq \C\setminus \overline{\D}\to \Omega_-$, we define
	\begin{equation*}
		\Sr(\Gamma)\coloneqq \int_{\D} \left|\frac{f''(z)}{f'(z)}\right|^2\, \dd x\, \dd y + \int_{\D^*} \left|\frac{g''(z)}{g'(z)}\right|^2\, \dd x\, \dd y + 4\pi\, \log\left| \frac{f'(0)}{g'(\infty)}\right|.
	\end{equation*}
	The functional $\Sr$ can be regarded as a Kähler potential on the connected component of the identity in the universal Teichmüller space \cite{TakhtajanTeo2006}. A Jordan curve $\Gamma$ for which $\Sr(\Gamma)<+\infty$ is called a Weil--Petersson curve. Several characterizations are proved in \cite{Bishop2025}, one of which is that $\Gamma$ has an arclength parametrization whose tangent vector lies in the Sobolev space $\W^{\frac{1}{2},2}(\Gamma)$. This class of curves is discussed in \Cref{sec:normalization} below. It is interesting to observe that the notion of a Weil--Petersson curve is the one-dimensional counterpart of the notion of weak immersions in a critical Sobolev space introduced for surfaces in \cite{R2014} and for higher-dimensional submanifolds in \cite{MR2025,MR2026}.
	
	In \cite{BBVPW2025}, the authors proved that the identity $4\VR(\Gamma)=\Sr(\Gamma)$ holds for any Jordan curve of class $\Cc^{5,\alpha}$ and that $\Sr(\Gamma)\geq 4\VR(\Gamma)$ for any Weil--Petersson curve. The main result of this work is to extend this identity to arbitrary Weil--Petersson curves.
	\begin{theorem}\label{th:continuity}
		The functional $\VR$ is continuous with respect to the Weil--Petersson topology: if $(\Gamma_j)_{j\in\N}$ is a sequence of Weil--Petersson curves converging to $\Gamma$ in the Weil--Petersson topology, then $\VR(\Gamma_j)\to \VR(\Gamma)$ as $j\to\infty$.
	\end{theorem}
	
	As a corollary, by following the proof of \cite[Corollary 5.12]{BBVPW2025} verbatim, we obtain the following result.
	\begin{corollary}\label{cor:equality}
		For any Weil--Petersson curve $\Gamma\subset \Chat$, we have $\Sr(\Gamma) = 4\VR(\Gamma)$.
	\end{corollary}
	
	We now make three remarks concerning the scope of this result.
	\begin{enumerate}
		\item As explained below, the ideas of the proof are reminiscent of those used in the analysis of weak immersions for (generalized) Willmore surfaces. This approach is rather different from the various tools developed in \cite{Bishop2025}. It would be interesting to understand whether these ideas can be adapted to similar problems, for instance in \cite{BV2023} or \cite{Bishop2025}.
		
		\item As explained in \cite{BBVPW2025}, Corollary \ref{cor:equality}, together with \cite[Theorem 8.1]{Wang2019}, implies that $4\VR(\Gamma)=\pi\IL(\Gamma)$, where $\IL$ is the Loewner energy that appears in probability theory and conformal field theory, see, for instance, \cite{W2022,W2024}. The present work provides a holographic interpretation of the Loewner energy in full generality.
		
		\item When the Jordan curve $\Gamma$ has a corner, it is not Weil--Petersson and, in particular, its Loewner energy is infinite. A first-order asymptotic expansion of the Loewner energy along a well-chosen approximating sequence has been obtained in \cite{JV2026}. It would be interesting to understand whether the present approach can provide the next terms.
	\end{enumerate}
	
	One may then ask whether there is a comparison between $\Ar$ and $\VR$. To do so, we introduce the exterior Grunsky operator $\G$ of a Jordan curve $\Gamma\subset \C$ as follows (see, for instance, \cite{TT2003}). We denote by $\Omega_+$ the compact connected component of $\C\setminus \Gamma$. Let $f\colon \overline{\D}\to \Omega_+$ be a uniformizing map. We consider the expansion 
	\begin{equation*}
		\log \frac{ f(z) - f(w) }{ z-w } = b_{0,0} - \sum_{m,n\geq 0} b_{m,n}\, z^m\, w^n.
	\end{equation*}
	We define $B_1\coloneqq (\sqrt{nm}\, b_{m,n})_{m,n\geq 1}$. By \cite{TakhtajanTeo2006}, the curve $\Gamma$ is Weil--Petersson if and only if $B_1$ is Hilbert--Schmidt:
	\begin{equation*}
		\|B_1\|_{\rm HS}^2 = \sum_{m,n\geq 1} nm\, |b_{m,n}|^2.
	\end{equation*}
	The universal Liouville action is given by the following formula \cite[Chapter 2, Theorem 3.8]{TakhtajanTeo2006}:
	\begin{equation}\label{eq:logdet}
		\Sr(\Gamma) = -12\pi\, \log\det(I - B_1^* B_1).
	\end{equation}
	We prove the following result.
	\begin{theorem}\label{th:comparison}
		Let $\Gamma\subset \C$ be a Weil--Petersson curve, and let $Y$ be a properly embedded minimal surface in $\Hd^3$ with $\dr Y = \Gamma$. Then we have
		\begin{equation*}
			\Ar(Y) \geq -2\pi -4\pi\, \|B_1\|_{\rm HS}^2 \geq -2\pi -\frac{1}{3}\, \Sr(\Gamma).
		\end{equation*}
		Moreover, there is equality if and only if $\Gamma$ is a round circle and $Y$ is a round half-sphere.
	\end{theorem}
	
	
	\textbf{Sketch of the proof of Theorem \ref{th:continuity}.} A canonical approximation scheme for Weil--Petersson curves was introduced in \cite[Corollary 1.5]{ViklundWang2020}. Therefore, it is sufficient to work first with analytic Jordan curves and then prove the continuity of $\VR$ along these special approximation sequences.
	
	If $\Gamma\subset \Chat$ is an analytic Jordan curve splitting $\Chat$ into two simply connected open sets $\Omega_+$ and $\Omega_-$, we consider two uniformizing maps $f\colon \D\to \Omega_+$ and $g\colon \D^*\to \Omega_-$.
	By integration by parts, we express the signed volume between the two Epstein surfaces induced by $\Gamma$ at the scale $\eps>0$ as the sum of two integrals over the images of these surfaces in \eqref{eq:volume-boundary-primitive}. In Lemma \ref{lm:decompo-sigmaf} and Lemma \ref{prop:exterior-epstein-flux}, we then compare the corresponding integrands $\sigma_f^+$ and $\sigma_g^-$ with their counterparts $\sigma_{\Id}^+$ and $\sigma_{\Id}^-$ in the case of the round circle $\Gamma=\s^1$: 
	\begin{equation*}
		\begin{cases}
			\sigma_f^+ = \sigma_{\Id}^+ + L_f^+ + R_f^+,\\[2mm]
			\sigma_g^- = \sigma_{\Id}^- + L_g^- + R_g^-.
		\end{cases}
	\end{equation*}
	The main observation is that, even though the signed volume between the two Epstein surfaces induced by $\s^1$ is $0$ because their images coincide, the two integrands $\sigma_{\Id}^-$ and $\sigma_{\Id}^+$ are not integrable. Therefore, when computing the sum of the two integrals, a cancellation occurs, which we record in \eqref{eq:round-integral-cancellation}. Hence, when computing the sum of the integrals of $\sigma_f^+$ and $\sigma_g^-$, one can restrict attention to the sum of the integrals of $(L_f^+ + R_f^+)$ and $(L_g^- + R_g^-)$. The terms $R_f^+$ and $R_g^-$ have exactly the integrability provided by the Weil--Petersson topology and can be regarded as lower-order terms. 
	
	To analyze the terms $L_f^+$ and $L_g^-$, we express them in Lemma \ref{lm:decompo-Lf} and Lemma \ref{lm:decompo-Lg} in terms of the radial derivatives of $\log |f'|$ and $\log|g'|$. We regard these logarithms as the harmonic extensions of their restrictions to $\s^1$. The key ingredient is then a monotonicity formula for a well-chosen $\W^{\frac{1}{2},2}$ scalar product in Lemma \ref{lem:Poisson-semigroup-monotonicity}. This formula is reminiscent of the monotonicity formulas for subharmonic functions and provides the key estimate needed to let the scale $\eps$ tend to $0$. We obtain an expression for $V(\Gamma)$ in \eqref{eq:analytic-finite-part-volume} that one can readily check to be continuous with respect to the Weil--Petersson topology.

	\textbf{Sketch of the proof of Theorem \ref{th:comparison}.} 
	The main ingredient in the proof is the identification of $\|B_1\|_{\rm HS}^2$ with a conformally invariant energy $\EOS$ introduced by O'Hara--Solanes \cite{OS2015}. This functional is obtained by a refined study of the Gauss--Bonnet formula for surfaces in $\Hd^3$ by Solanes \cite{S2010}. For minimal surfaces, the second term in \eqref{eq:AM-functional} is bounded below by the total extrinsic curvature. To obtain Theorem \ref{th:comparison}, we compare the Hilbert--Schmidt norm of $B_1$ with its log determinant \eqref{eq:logdet}.
	
	\textbf{Organization.} In Section \ref{sec:normalization}, we introduce the notions used throughout the manuscript: we define the Weil--Petersson topology, give the parametrization of Epstein surfaces, and discuss the Poisson kernel and the Dirichlet-to-Neumann operator. In Section \ref{sec:epstein-flux}, we discuss integration by parts and decompose the integrands. In Section \ref{sec:cutoff-volume-formula}, we pass to the limit $\eps\to 0$ to express the signed volume between Epstein surfaces in terms of the conformal maps. In Section \ref{sec:continuity}, we prove Theorem \ref{th:continuity}. In Section \ref{sec:comparison}, we prove Theorem \ref{th:comparison}.
	
	\textbf{Use of AI.} The author used ChatGPT 5.6 to assist with conceptualization and computations. All mathematical validation, final proof decisions, and final wording remain the sole responsibility of the human author.
	
	\textbf{Acknowledgments.} This project is funded by the Swiss National Science Foundation, project SNF 200020\textunderscore 219429. I thank Yilin Wang, Catherine Wolfram, and Viola Giovannini for many enlightening and motivating discussions. I also thank Tristan Rivière for his support.
	
	\section{Weil--Petersson curves and Epstein surfaces}\label{sec:normalization}
	
	Throughout the article, we write $\dd^2 z = \dd x\, \dd y$ for the volume form on $\C$. We write $\Chat =\C\cup\{\infty\}$, denote by $\D$ the open unit disk centered at the origin and by $\D_r$ the open disk of radius $r>0$ centered at the origin, and set $\D^*\coloneqq \C\setminus \overline{\D}$.
	
	\subsection{Uniformizing maps for Jordan curves}\label{sec:uniformisation}
	
	Let $\Gamma\subset \Chat$ be a Jordan curve and consider the two connected components $\Chat\setminus\Gamma \eqqcolon \Omega_+\sqcup\Omega_-$, where $\infty\in \Omega_-$ and $0\in \Omega_+$. In particular, $\Omega_+$ is bounded in $\C$. We choose two conformal maps $f\colon \D\to \Omega_+$ and $g\colon \D^*\to \Omega_-$. Up to normalizing $\Gamma$ using Möbius transformations, we may impose the following conditions:
	\begin{equation}\label{eq:normalisation}
		\begin{cases} 
			f(-1)=-1,\qquad f(-\ii)=-\ii,\qquad f(1)=1,\\[2mm]
			g(-1)=-1,\qquad g(-\ii)=-\ii,\qquad g(1)=1.
		\end{cases} 
	\end{equation}
	We denote the pre-Schwarzian and Schwarzian derivatives by
	\begin{equation*}
		\Nc(f) \coloneqq \frac{f''}{f'}, \qquad \text{ and }\qquad \Sc(f) \coloneqq \left(\frac{f''}{f'}\right)' - \frac{1}{2}\, \left( \frac{f''}{f'}\right)^2.
	\end{equation*}
	
	\subsection{Weil--Petersson topology} A Jordan curve $\Gamma\subset \Chat$ is said to be a Weil--Petersson curve (WP curve) if 
	\begin{equation*}
		\int_{\D} \left|\Nc(f)\right|^2\, \dd^2 z  + \int_{\D^*}\left|\Nc(g)\right|^2\, \dd^2 z <+\infty. 
	\end{equation*}
	Equivalently, by \cite[Chapter 2, Theorem 1.12]{TakhtajanTeo2006}, we have
	\begin{equation*}
		\int_{\D} \left|\Sc(f)(z)\right|^2 (1-|z|^2)^2\, \dd^2 z + \int_{\D^*} \left|\Sc(g)(z)\right|^2 (|z|^2-1)^2\, \dd^2 z <+\infty.
	\end{equation*}
	By \cite[Theorem 1.1]{Bishop2025}, this is equivalent to $\Gamma$ having an arclength parametrization in $\W^{\frac{3}{2},2}(I;\Chat)$ for some interval $I$. The convergence of a sequence $(\Gamma_k)_{k\in\N}$ of WP curves to a WP curve $\Gamma$ is characterized by the following properties:
	\begin{equation}\label{eq:WP-convergence-data}
		\|\Nc(f_k)-\Nc(f)\|_{L^2(\D)} + \|\Nc(g_k)-\Nc(g)\|_{L^2(\D^*)}  \xrightarrow[k\to \infty]{} 0.
	\end{equation}
	By \cite[Lemma A.1 and A.5]{TakhtajanTeo2006}, \eqref{eq:WP-convergence-data} is equivalent to 
	\begin{equation}\label{eq:WP-convergence-data2}
		\int_{\D} \left|\Sc(f_k)(z) - \Sc(f)(z)\right|^2 (1-|z|^2)^2\, \dd^2 z + \int_{\D^*} \left|\Sc(g_k)(z) - \Sc(g)(z)\right|^2 (|z|^2-1)^2\, \dd^2 z \xrightarrow[k\to \infty]{} 0.
	\end{equation}
	
	Let $u_+$ and $u_-$ be the boundary traces of $\log|f'|$ and $\log|g'|$. Then $\log|f'|$ and $\log|g'|$ are the harmonic extensions of $u_+$ and $u_-$, respectively. Consequently,
	\begin{equation}\label{eq:trace-energies}
		[u_+]_{\dot{\W}^{\frac{1}{2},2}(\s^1)}=\int_{\D}|\Nc(f)|^2\,\dd^2 z,
		\qquad
		[u_-]_{\dot{\W}^{\frac{1}{2},2}(\s^1)}=\int_{\D^*}|\Nc(g)|^2\,\dd^2 z,
	\end{equation}
	where, for a function $v\colon \s^1\to\C$, we have written \cite[Theorem 2.5]{A2010}
	\begin{equation*}
		[v]_{\dot{\W}^{\frac{1}{2},2}(\s^1)} \coloneqq  \frac{1}{2\pi} \int_{\s^1\times \s^1} \frac{|v(\xi)-v(\zeta)|^2}{|\xi-\zeta|^2}\, \dd \xi\, \dd \zeta.
	\end{equation*}
	For a WP curve, \cite[Chapter 2, Theorem 1.6]{TakhtajanTeo2006} gives
	\begin{equation}\label{eq:boundary-Schwarzian}
		\begin{cases} 
			(1-|z|^2)|\Nc(f)(z)| + (1-|z|^2)^2|\Sc(f)(z)| \xrightarrow[|z|\to 1^-]{} 0\\[2mm]
			(|z|^2-1)|\Nc(g)(z)| + (|z|^2-1)^2|\Sc(g)(z)| \xrightarrow[|z|\to 1^+]{} 0.
		\end{cases} 
	\end{equation}
	
	With the normalization \eqref{eq:normalisation}, the universal Liouville action is defined in \cite[Definition 2.13]{TakhtajanTeo2006} by
	\begin{equation}\label{eq:Liouville-Dirichlet-form}
		\Sr(f,g)= \int_\D|\Nc(f)|^2\, \dd^2 z+ \int_{\D^*}|\Nc(g)|^2\, \dd^2 z + 4\pi\log \frac{|f'(0)|}{|g'(\infty)|}.
	\end{equation}

	\subsection{Approximation of WP curves} Let $\Gamma$ be a WP curve with uniformizing maps $f$ and $g$ as in \Cref{sec:uniformisation}.
	For $a\in(0,1)$, we define
	\begin{equation}\label{eq:canonical-radial-map}
		f_a(z)=a^{-1}f(az),\qquad \Omega_a=f_a(\D),\qquad\Gamma_a=f_a(\s^1).
	\end{equation}
	Let $g_a\colon \D^*\to\Chat\setminus\overline{\Omega_a}$ be the normalized complementary map. For $a_j=\frac{j-1}{j}$ with $j\geq 2$, we denote 
	\begin{equation}\label{eq:approximation}
		f_j = f_{a_j}, \qquad g_j = g_{a_j}, \qquad \Gamma_j = \Gamma_{a_j}.
	\end{equation} 
	Each $\Gamma_j$ is analytic. Its normalized conformal maps converge to those of $\Gamma$ in the Weil--Petersson topology described in \eqref{eq:WP-convergence-data}.
	For the maps in \eqref{eq:canonical-radial-map}, we have
	\begin{equation}\label{eq:canonical-P-Q}
		\Nc(f_a)(z)=a\, \Nc(f)(az), \qquad \text{and }\qquad  \Sc(f_a)(z) = a^2\, \Sc(f)(az).
	\end{equation}
	Hence, $\Nc(f_a)\to \Nc(f)$ in $\Ll^2(\D)$ and $\Sc(f_a)\to \Sc(f)$ in $\Ll^2\left(\D,(1-|z|^2)^2\, \dd^2 z \right)$. By \cite[Corollary 1.5]{ViklundWang2020}, we have $\Gamma_a\to \Gamma$ as $a\to 1^-$ in the Weil--Petersson topology.
	
	We will frequently use the following result.
	\begin{theorem}[Theorem 1 in \cite{BL1983}]\label{th:Brezis-Lieb}
		Let $(X,\mu)$ be a measure space and $(f_k)_{k\in\N},f\in \Ll^1(X)$. Assume that $f_k\to f$ a.e.\ in $X$. Then we have $\|f_k\|_{\Ll^1(X,\mu)}\xrightarrow[k\to\infty]{} \|f\|_{\Ll^1(X,\mu)}$ if and only if $f_k\to f$ in $\Ll^1(X,\mu)$.
	\end{theorem}
	
	\subsection{Epstein surfaces}\label{sec:epstein}
	In this section, we record the explicit parametrization of Epstein surfaces provided in \cite[Section 3.3]{BBVPW2025}.

	Let $\Omega\subset \Chat$ be a simply connected bounded open set, and let $e^{\vp}\, |\dd z|^2$ be the associated hyperbolic metric. The Epstein surface of $\Omega$ is the map $\Ep_{\Omega}\colon z\in \Omega\mapsto (Z,\xi)\in \C\times (0,+\infty) = \Hd^3$ defined by
	\begin{equation*}
		\xi \coloneqq \frac{ 2 e^{-\vp/2} }{ 1+|\vp_{\bz}|^2\, e^{-\vp} }, \qquad \text{ and }\qquad Z \coloneqq  z + \frac{2\, \vp_{\bz}\, e^{-\vp} }{ 1+|\vp_{\bz}|^2\, e^{-\vp} }.
	\end{equation*}
	The Euclidean unit normal vector of $\Ep_{\Omega}$ at $z\in\Omega$ is given by 
	\begin{equation}\label{eq:normal-Epstein}
		\vec{\eta} \coloneqq \left( \frac{ 2\, \vp_{\bz}\, e^{-\vp/2} }{ 1+|\vp_{\bz}|^2\, e^{-\vp} } , \frac{ 1 - |\vp_{\bz}|^2\, e^{-\vp}}{1+|\vp_{\bz}|^2\, e^{-\vp}} \right).
	\end{equation}
	By \cite[Lemma 3.16]{BBVPW2025}, we have 
	\begin{equation}\label{eq:coordinates-Epstein}
		\begin{cases} 
			\displaystyle \xi(f(z)) = \frac{|f'(z)|\, (1-|z|^2)}{1+ \left|\overline{z}- \frac{1}{2}\, \Nc(f)(z)\, (1-|z|^2) \right|^2 }, \\[5mm]
			\displaystyle Z(f(z)) = f(z) + \frac{ \left(z - \frac{1}{2}\, \overline{\Nc(f)(z)}\, (1-|z|^2) \right)\, f'(z)\, (1-|z|^2) }{1 + \left|\overline{z}- \frac{1}{2}\, \Nc(f)(z)\, (1-|z|^2) \right|^2 }.
		\end{cases} 
	\end{equation}
	The formula \eqref{eq:normal-Epstein} can also be written as
	\begin{equation}\label{eq:normal-Epstein2}
		\vec{\eta}(f(z)) \coloneqq \left( \frac{ 2\, \frac{|f'(z)|}{\overline{f'(z)} }\, \left(z - \frac{1}{2}\, \overline{\Nc(f)(z)}\, (1-|z|^2)\right)}{ 1+ \left|z - \frac{1}{2}\, \overline{\Nc(f)(z)}\, (1-|z|^2)\right|^2 } , \frac{ 1 - \left|z - \frac{1}{2}\, \overline{\Nc(f)(z)}\, (1-|z|^2)\right|^2 }{1+\left|z - \frac{1}{2}\, \overline{\Nc(f)(z)}\, (1-|z|^2)\right|^2 } \right).
	\end{equation}
	For instance, if $\Gamma=\s^1$, we can choose $f(z)=z$ and obtain 
	\begin{equation}\label{eq:Epstein-round}
		\begin{cases} 
			\displaystyle \xi_{\Id}(z) = \frac{1-|z|^2}{1+ |z|^2 }, \\[5mm]
			\displaystyle Z_{\Id}(z) =  \frac{ 2z }{1 + |z|^2 }.
		\end{cases} 
	\end{equation}
	
	\subsection{Poisson kernel and Dirichlet-to-Neumann operator}\label{sec:Poisson}
	
	For $0\leq r<1$, we define the Poisson kernel by
	\begin{equation*}
		\forall \alpha\in [0,2\pi],\qquad p_r(\alpha) =\frac{1-r^2}{1-2r\cos\alpha+r^2}.
	\end{equation*}
	For a function $u\in \W^{\frac{1}{2},2}(\s^1,\R)$, its Poisson extension to the circle of radius $r$ is defined by
	\begin{equation}\label{eq:Poisson-operator-definition}
		(\Pcal_ru)(\theta) \coloneqq \frac{1}{2\pi}\int_0^{2\pi} p_r(\theta-\varphi)u(\varphi)\,\dd\varphi.
	\end{equation}
	The function $U_u(re^{\ii\theta})\coloneqq(\Pcal_ru)(\theta)$ is the harmonic extension of $u$ to $\D$. 
	We expand $u$ in a Fourier series:
	\begin{equation*}
		u(\theta)=\sum_{n\in\Z}u_ne^{\ii n\theta}, \qquad \text{ with }\qquad u_n=\frac{1}{2\pi}\int_0^{2\pi} u(\theta)e^{-\ii n\theta}\,\dd\theta.
	\end{equation*}
	Then $\Pcal_r u$ is given by
	\begin{equation}\label{eq:Poisson-Fourier-definition}
		\Pcal_r u =\sum_{n\in\Z}r^{|n|}u_ne^{\ii n\theta}.
	\end{equation}
	The Dirichlet-to-Neumann operator on $\s^1$ is defined by
	\begin{equation}\label{eq:Lambda-geometric-definition}
		\Lambda u \coloneqq \left.\partial_rU_u(re^{\ii\theta})\right|_{r=1} = \sum_{n\in\mathbb Z}|n|u_ne^{in\theta}.
	\end{equation}
	Therefore, the operator $\Lambda\colon \W^{\frac{1}{2},2}(\s^1)\to \W^{-\frac{1}{2},2}(\s^1)$ is bounded\footnote{Indeed, $\sum_{n\neq 0} |n|^{-1} |(\Lambda u)_n|^2 = \sum_{n\neq 0} |n|\, |u_n|^2$.}.
	The duality between $\W^{\frac{1}{2},2}(\s^1)$ and $\W^{-\frac{1}{2},2}(\s^1)$ is given by
	\begin{equation}\label{eq:Dirichlet-form-definition}
		\Er(u,v)\coloneqq \scal{\Lambda u}{v}_{\W^{-\frac{1}{2},2},\W^{\frac{1}{2},2}} =2\pi\sum_{n\in\mathbb Z}|n|u_n\overline{v_n}.
	\end{equation}
	The family $(\Pcal_r)_{0<r<1}$ satisfies $\Pcal_r\Pcal_\rho=\Pcal_{r\rho}$. 
	In the additive variable \(s=-\log r\), this family becomes the Poisson semigroup, and \(-\Lambda\) is its infinitesimal generator \cite[Chapter II, Section 2]{SteinTopics1970}:
	\begin{equation}\label{eq:Poisson-semigroup-generator}
		\Pcal_{e^{-s}}=e^{-s\Lambda}, \qquad \text{ and }\qquad 
		-\left.\frac{\dd}{\dd s}\mathcal P_{e^{-s}}u \right|_{s=0} =\Lambda u.
	\end{equation}

	\section{First fundamental form of Epstein surfaces}\label{sec:epstein-flux}
	
	Throughout this section, we fix a WP curve $\Gamma$ and use the notation of Section \ref{sec:normalization}. 
	
	\subsection{Integration by parts}\label{sec:IPP}
	
	We work in upper half-space coordinates
	\begin{equation*}
		\mathbb H^3=\{(X,Y,\xi):\xi>0\}, \qquad \text{and }\qquad \dd\vol_{\Hd^3} =\xi^{-3}\dd X\wedge\dd Y\wedge\dd\xi.
	\end{equation*}
	For the rest of the paper, we fix a cutoff function $\eta\colon [0,\infty)\to[0,1]$ such that $\eta'\geq 0$, $\eta=0$ on $[0,1]$, and $\eta=1$ on $[2,\infty)$. We define
	\begin{equation}\label{eq:p-epsilon}
		p_\eps(t)\coloneqq \int_0^t\eta(s/\eps)\,\frac{\dd s}{s^3},
		\qquad \text{ and }\qquad 
		\alpha_\eps \coloneqq p_\eps(\xi)\,\dd X\wedge\dd Y.
	\end{equation}
	Since $p_\eps'(\xi)=\eta(\xi/\eps)\xi^{-3}$, we have
	\begin{equation}\label{eq:primitive-volume}
		\dd\alpha_\eps
		= p_\eps'(\xi)\, \dd\xi\wedge\dd X\wedge\dd Y = \frac{\eta(\xi/\eps)}{\xi^3}\, \dd X\wedge\dd Y \wedge\dd\xi
		=\eta(\xi/\eps)\, \dd\vol_{\Hd^3}.
	\end{equation}
	By \cite[Remark 4.3]{BBVPW2025}, there exists a map $\Phi_\Gamma\colon \overline{\Hd^3}\to \overline{\Hd^3}$ differentiable in $\Hd^3$, with $\Phi_{\Gamma} = \Ep_{\Omega_-}$ on $\Omega_-$ and $\Phi_{\Gamma} = \Ep_{\Omega_+}$ on $\Omega_+$. Integrating by parts and using \eqref{eq:primitive-volume}, we obtain
	\begin{equation}\label{eq:volume-boundary-primitive}
		V_\eps(\Gamma) \coloneqq\int_{\mathbb H^3}\Phi_\Gamma^*\bigl(\eta(\xi/\eps)\, \dd \vol_{\Hd^3}\bigr) = -\sum_{\sigma\in\{+,-\}} \int_{\Omega_\sigma} \Ep_{\Omega_\sigma}^*\alpha_\eps.
	\end{equation}
	The sign is negative because the boundary orientation induced by $\dd x\wedge\dd y\wedge\dd\xi$ is opposite to the complex orientation on $\{\xi=0\}$.

	\subsection{Horizontal slices}
	
	We now compute $\Ep_{\Omega_+}^*(\xi^{-2}\dd X\wedge\dd Y)$. For $z\in \D$, we denote
	\begin{equation}\label{eq:flux-notation}
		\rho_+ \coloneqq 1-|z|^2, \qquad \text{ and }\qquad \psi_f \coloneqq \overline{z}-\frac{\rho_+}{2}\, \Nc(f)(z).
	\end{equation}
	By a slight abuse of the notation introduced in \Cref{sec:epstein}, we write $(Z,\xi)=\Ep_{\Omega_+}\circ f$.
	\begin{lemma}\label{prop:exact-epstein-flux}
		With the above notation, we have
		\begin{equation}\label{eq:exact-epstein-flux}
			(\Ep_{\Omega_+}\circ f)^*\bigl(\xi^{-2}\dd X\wedge\dd Y \bigr) = \sigma_f\, \dd x\wedge \dd y,
		\end{equation}
		where 
		\begin{equation}\label{eq:def-sigma}
			\sigma_f \coloneqq \frac{1-|\psi_f|^2}{1+|\psi_f|^2} \left(\frac{4}{\rho_+^2}-\frac{\rho_+^2}{4}|\Sc(f)|^2\right).
		\end{equation}
	\end{lemma}
	
	\begin{proof}
		We denote $E\coloneqq \Ep_{\Omega_+}\circ f$. With the notation \eqref{eq:flux-notation}, the formulas \eqref{eq:coordinates-Epstein} can be rewritten as 
		\begin{equation}\label{eq:epstein-f-coordinates}
			\xi=\frac{|f'|\, \rho_+ }{1+|\psi_f|^2}, \qquad \text{ and }\qquad  Z=f+\frac{f'\, \rho_+\,\overline{\psi_f}}{1+|\psi_f|^2}.
		\end{equation}
		The vertical component of the oriented Euclidean unit normal is given by \eqref{eq:normal-Epstein2}:
		\begin{equation}\label{eq:vertical-normal}
			\vec{\eta}_3=\frac{1-|\psi_f|^2}{1+|\psi_f|^2}.
		\end{equation}
		The signed area form induced by $E$ is obtained by combining Equations (3.4) and (3.5) of \cite{BBVPW2025}:
		\begin{equation}\label{eq:vol-Epstein}
			\begin{aligned} 
				\dd a &=\frac{1}{4} \left( 2-2\, \frac{\rho_+^2}{4}\, |\Sc(f)| \right) \left( 2+2\, \frac{\rho_+^2}{4}\, |\Sc(f)| \right)\, \frac{4}{\rho_+^2}\,\dd x\wedge \dd y\\[2mm]
				&=\left(1- \frac{\rho_+^4}{16}\, |\Sc(f)|^2 \right)\frac{4}{\rho_+^2}\, \dd x\wedge \dd y\\[2mm]
				& =\left(\frac{4}{\rho_+^2}-\frac{\rho_+^2}{4}|\Sc(f)|^2\right) \, \dd x\wedge \dd y.
			\end{aligned} 
		\end{equation}
		Therefore, we have 
		\begin{equation*}
			E^*(\dd X\wedge\dd Y) =\bigl( (\dr_x E_1)(\dr_y E_2) - (\dr_y E_1)(\dr_x E_2) \bigr)\,\dd x\wedge\dd y =  \left( \dr_x E\times \dr_y E\right)_3\,\dd x\wedge\dd y.
		\end{equation*} 
		Since $\dr_x E\times\dr_y E$ is normal to $\Span(\dr_x E,\dr_y E)$ and $\vec{\eta}$ is the Gauss map of $E$, we obtain 
		\begin{equation*} 
			E^*(\xi^{-2}\, \dd X\wedge\dd Y) = \xi^{-2}\, \scal{ \dr_x E \times \dr_y E}{\vec{\eta}}_{\R^3}\, \vec{\eta}_3\,\dd x\wedge\dd y = \vec{\eta}_3\, \dd a.
		\end{equation*}
		We conclude by substituting \eqref{eq:vol-Epstein} into the preceding equation.
	\end{proof}
	
	We now expand \eqref{eq:exact-epstein-flux} into three parts: the leading-order term corresponding to the round case, a second term that can, roughly speaking, be viewed as linear in $\Nc$, and a remainder that is quadratic in $\rho_+^2\, |\Sc|$ and $\rho_+|\Nc|$. The latter two terms converge to $0$ near the boundary by \eqref{eq:boundary-Schwarzian}. When $\Gamma$ is a round circle, we can choose $f(z)=z$ and obtain $\psi_f=\overline{z}$ and $\Sc(f)=0$. Therefore, we have
	\begin{equation}\label{eq:sigma-id-interior}
		\sigma_{\Id}^+ = \frac{4}{(1+|z|^2)(1-|z|^2)} = \frac{4}{\rho_+\, (2-\rho_+)}.
	\end{equation}
	
	\begin{lemma}\label{lm:decompo-sigmaf}
		There exist functions $L_f^+\colon\D\to \C$ and $R_f^+\colon \D\to \C$ such that 
		\begin{equation*}
			\sigma_f =\sigma_{\Id}^+ + L_f^+ + R^+_f,
		\end{equation*}
		where 
		\begin{enumerate}
			\item the function $L_f^+$ is given by 
			\begin{equation*}
				L_f^+ \coloneqq \frac{ 8\, \Re\left(z\, \Nc(f)\right) }{ \rho_+\, (2-\rho_+)^2 },
			\end{equation*}
			
			\item the remainder $R_f^+$ satisfies the pointwise estimate
			\begin{equation*}
				|R_f^+| \leq C_0\, \left(|\Nc(f)|^2 + \rho_+^2\, |\Sc(f)|^2 \right) \qquad \text{ in }\D .
			\end{equation*}
		\end{enumerate}
		The constant $C_0$ depends only on the rate of convergence in \eqref{eq:boundary-Schwarzian} and can be chosen uniformly along sequences of curves converging in the WP topology. 
	\end{lemma}
	\begin{proof} 
		We have 
		\begin{equation}\label{eq:psi-square}
			\begin{aligned} 
				|\psi_f|^2 &=\left(\overline z-\frac{\rho_+}{2}\, \Nc(f)\right) \left(z-\frac{\rho_+}{2}\, \overline{\Nc(f)} \right)\\[2mm]
				&=|z|^2-\frac{\rho_+}{2}\left(z\, \Nc(f)+\overline{z\, \Nc(f)}\right) +\frac{\rho_+^2}{4}\, |\Nc(f)|^2\\[2mm]
				&=|z|^2-\rho_+\, \Re(z\, \Nc(f))+\frac{\rho_+^2}{4}\, |\Nc(f)|^2.
			\end{aligned} 
		\end{equation}
		In particular, we obtain 
		\begin{equation*}
			1-|\psi_f|^2 = \rho_+ + \rho_+\, \Re(z\, \Nc(f))-\frac{\rho_+^2}{4}\, |\Nc(f)|^2 =\rho_+\, \tilde{\psi}, 
		\end{equation*}
		where
		\begin{equation}\label{eq:tilde-psi}
			\tilde{\psi}\coloneqq 1 + \Re(z\, \Nc(f)) - \frac{\rho_+}{4}\, |\Nc(f)|^2.
		\end{equation}
		We also have
		\begin{equation*}
			1+|\psi_f|^2 = 2-\rho_+\,\tilde{\psi}.
		\end{equation*}
		From \eqref{eq:def-sigma}, we have
		\begin{equation*}
			\sigma_f = \frac{4}{\rho_+} \, \frac{  \tilde{\psi} }{ 2- \rho_+\,\tilde{\psi}} - \frac{ \tilde{\psi}}{2-\rho_+\,\tilde{\psi}}\, \frac{\rho_+^3}{4}\, |\Sc(f)|^2.
		\end{equation*}
		From \eqref{eq:def-sigma}, we obtain 
		\begin{equation*}
			\begin{aligned}
				\sigma_f - \sigma_{\Id} & = \frac{4}{\rho_+}\left(\frac{  \tilde{\psi} }{ 2- \rho_+\,\tilde{\psi}} - \frac{1}{2-\rho_+}  \right)  - \frac{ \tilde{\psi}}{2-\rho_+\,\tilde{\psi}}\, \frac{\rho_+^3}{4}\, |\Sc(f)|^2 \\[2mm]
				& = \frac{4}{\rho_+}\, \frac{ \tilde{\psi}\, (2-\rho_+) - 2+\rho_+\,\tilde{\psi} }{ (2-\rho_+\, \tilde{\psi})\, (2-\rho_+) } - \frac{ \tilde{\psi}}{2-\rho_+\,\tilde{\psi}}\, \frac{\rho_+^3}{4}\, |\Sc(f)|^2 \\[2mm]
				& = \frac{8}{\rho_+}\, \frac{ \tilde{\psi}-1 }{ (2-\rho_+\, \tilde{\psi})\, (2-\rho_+) } - \frac{ \tilde{\psi}}{2-\rho_+\,\tilde{\psi}}\, \frac{\rho_+^3}{4}\, |\Sc(f)|^2 .
			\end{aligned}
		\end{equation*}
		Using the formula \eqref{eq:tilde-psi}, we obtain 
		\begin{equation*}
			\begin{aligned} 
				\sigma_f - \sigma_{\Id} & = \frac{8}{\rho_+}\, \frac{ \Re\left(z\, \Nc(f)\right) }{ (1+|\psi_f|^2)\, (2-\rho_+) }  -  \frac{2\, |\Nc(f)|^2 }{ (1+|\psi_f|^2)\, (2-\rho_+) } - \frac{ \tilde{\psi}}{1+|\psi_f|^2}\, \frac{\rho_+^3}{4}\, |\Sc(f)|^2 \\[2mm]
				& = \frac{8\, \Re\left(z\, \Nc(f)\right) }{\rho_+ (2-\rho_+)^2 } + \frac{8}{\rho_+}\, \left[ \frac{ \Re\left(z\, \Nc(f)\right) }{ (1+|\psi_f|^2)\, (2-\rho_+) } -\frac{ \Re\left(z\, \Nc(f)\right) }{ (2-\rho_+)^2 } \right] -  \frac{2\, |\Nc(f)|^2 }{ (1+|\psi_f|^2)\, (2-\rho_+) } - \frac{ \tilde{\psi}}{1+|\psi_f|^2}\, \frac{\rho_+^3}{4}\, |\Sc(f)|^2 \\[2mm]
				& = \frac{8\, \Re\left(z\, \Nc(f)\right) }{\rho_+ (2-\rho_+)^2 } + \frac{8}{\rho_+}\,  \frac{ ( 1-\rho_+ - |\psi_f|^2 )\, \Re\left(z\, \Nc(f)\right) }{ (1+|\psi_f|^2)\, (2-\rho_+)^2 }  -  \frac{2\, |\Nc(f)|^2 }{ (1+|\psi_f|^2)\, (2-\rho_+) } - \frac{ \tilde{\psi}}{1+|\psi_f|^2}\, \frac{\rho_+^3}{4}\, |\Sc(f)|^2 \\[2mm]
				& = \frac{8\, \Re\left(z\, \Nc(f)\right) }{\rho_+ (2-\rho_+)^2 } + \frac{8\, (\tilde{\psi}-1)\, \Re\left(z\, \Nc(f)\right) }{ (1+|\psi_f|^2)\, (2-\rho_+)^2 }  -  \frac{2\, |\Nc(f)|^2 }{ (1+|\psi_f|^2)\, (2-\rho_+) } - \frac{ \tilde{\psi}}{1+|\psi_f|^2}\, \frac{\rho_+^3}{4}\, |\Sc(f)|^2
			\end{aligned} 
		\end{equation*}
		Moreover, we have $1+|\psi_f|^2\geq 1$ and $2-\rho_+\geq 1$. Using \eqref{eq:boundary-Schwarzian}, the last two terms on the right-hand side can be bounded as follows:
		\begin{equation*}
			\frac{2\, |\Nc(f)|^2 }{ (1+|\psi_f|^2)\, (2-\rho_+) } + \frac{ |\tilde{\psi} | }{1+|\psi_f|^2}\, \frac{\rho_+^3}{4}\, |\Sc(f)|^2  \leq C\, \left( |\Nc(f)|^2 + \rho_+^2\, |\Sc(f)|^2 \right).  
		\end{equation*}
		Using \eqref{eq:tilde-psi} and \eqref{eq:boundary-Schwarzian}, we also obtain 
		\begin{equation*}
			\left| \frac{8\, (\tilde{\psi}-1)\, \Re\left(z\, \Nc(f)\right) }{ (1+|\psi_f|^2)\, (2-\rho_+) } \right| \leq C\, |\Nc(f)|^2.\qedhere 
		\end{equation*}
	\end{proof} 
	
	We can perform analogous computations with $g$ instead of $f$ and $\rho_- = |z|^2-1=-\rho_+$ instead of $\rho_+$ (which entails some changes of sign). We obtain the following result.
	\begin{lemma}\label{prop:exterior-epstein-flux}
		If $z\in \D^*$, we denote
		\begin{equation}\label{eq:def-psig}
			\rho_-\coloneqq |z|^2-1,\qquad \text{ and }\qquad \psi_g \coloneqq \overline{z}+\frac{\rho_-}{2}\, \Nc(g).
		\end{equation}
		Then
		\begin{equation}\label{eq:exterior-epstein-flux}
			(\Ep_{\Omega_-}\circ g)^*\bigl(\xi^{-2}\dd X\wedge\dd Y\bigr) = \sigma_g^-\, \dd x\wedge \dd y,
		\end{equation}
		with 
		\begin{equation}
			\sigma_g^- \coloneqq \frac{1-|\psi_g|^2}{1+|\psi_g|^2} \left(\frac{4}{\rho_-^2}-\frac{\rho_-^2}{4}|\Sc(g)|^2\right).
		\end{equation}
		For the identity map, we set $\sigma_{\Id}^-\coloneqq -4/[\rho_-(2+\rho_-)]$. There exist functions $L_g^-\colon\D^*\to \R$ and $R_g^-\colon \D^*\to \R$ such that 
		\begin{equation*}
			\sigma_g^- =\sigma_{\Id}^- + L_g^- + R_g^-,
		\end{equation*}
		where 
		\begin{enumerate}
			\item the function $L_g^-$ is given by 
			\begin{equation*}
				L_g^- \coloneqq -\frac{ 8\, \Re\left(z\, \Nc(g)\right) }{ \rho_-\, (2+\rho_-)^2 },
			\end{equation*}
			
			\item the remainder $R_g^-$ satisfies the pointwise estimate
			\begin{equation*}
				|R_g^-| \leq C\, \left(|\Nc(g)|^2 + \rho_-^2\, |\Sc(g)|^2 \right).
			\end{equation*}
		\end{enumerate}
	\end{lemma}
	
	By the definition of the Weil--Petersson topology, we have
	\begin{equation}\label{eq:global-remainders-L1}
		\begin{cases} 
			R_f^+\in \Ll^1(\D),\qquad \text{ and }\qquad  R_g^-\in \Ll^1(\D^*),\\[2mm]
			\text{ the maps }\Gamma\mapsto (R_f^+,R_g^-)\in \Ll^1(\D)\times \Ll^1(\D^*) \text{ are continuous for the WP topology}.
		\end{cases} 
	\end{equation}
	
	\section{The cutoff volume near the conformal boundary}\label{sec:cutoff-volume-formula}

	\subsection{The cutoff function}
	
	We consider the cutoff function $\eta$ from \eqref{eq:p-epsilon}. For $x>0$, we denote
	\begin{equation}\label{eq:defJ}
		J(x)=\int_0^x\eta(s)\,\frac{\dd s}{s^3} .
	\end{equation}
	Since $\eta=0$ on $[0,1]$ and $\eta=1$ on $[2,\infty)$, the function $\eta(s)s^{-3}$ is integrable on $(0,\infty)$. Let
	\begin{equation}\label{eq:jinfty}
		j_{\infty} \coloneqq \int_0^\infty\eta(s)\, \frac{\dd s}{s^3} =\lim_{x\to\infty} J(x).
	\end{equation}
	We have
	\begin{equation}\label{eq:cutoff-tail}
		j_{\infty}-J(x)=\int_x^\infty\eta(s)\, \frac{\dd s}{s^3}.
	\end{equation}
	In particular, because $\eta(s)=1$ for $s\geq2$, we have
	\begin{equation}\label{eq:large-J}
		\forall x\geq 2,\qquad j_{\infty}-J(x)=\frac{1}{2x^2}.
	\end{equation}
	We define
	\begin{equation}\label{eq:cutoff-profile}
		\begin{cases} 
			W(x) \coloneqq x^2\bigl(j_{\infty}-J(x)\bigr), \\[2mm]
			w(y) \coloneqq W(e^y).
		\end{cases} 
	\end{equation}

	\begin{lemma}\label{lem:cutoff-function-properties}
		The function $w$ is nondecreasing. We have the following limits:
		\begin{equation}\label{eq:limits-w}
			\begin{cases}
				w(y) \xrightarrow[y\to -\infty]{} 0, \\[2mm]
				\displaystyle w(y) \xrightarrow[y\to +\infty]{} \frac{1}{2}.
			\end{cases}
		\end{equation}
		Moreover, we have $w'\in \Ll^1\cap \Ll^{\infty}(\R)$. Its integral is given by 
		\begin{equation}\label{eq:integral-w}
			\int_{\R} w'(y)\,\dd y=\frac{1}{2}.
		\end{equation}
		For every $a\in\R$, we have
		\begin{equation}\label{eq:scale-shift}
			\int_{\R} \bigl(w(a-x)-w(-x)\bigr)\,\dd x=\frac{a}{2}.
		\end{equation}
	\end{lemma}
	
	\begin{proof}
		By \eqref{eq:cutoff-tail}, the function $W$ in \eqref{eq:cutoff-profile} is given by
		\begin{equation*}
			W(x) =x^2\int_x^\infty\eta(s)\,\frac{\dd s}{s^3}.
		\end{equation*}
		Using $J'(x)=\eta(x)x^{-3}$, we obtain
		\begin{equation*}
			W'(x) = 2x \int_x^{\infty} \eta(s)\, \frac{\dd s}{s^3} - \frac{\eta(x)}{x}.
		\end{equation*}
		Since $\eta$ is nondecreasing, we have 
		\begin{equation*}
			W'(x) \geq 2x \, \eta(x)\int_x^\infty \frac{\dd s}{s^3} - \frac{\eta(x)}{x} =0.
		\end{equation*}
		Hence we have $w'(y)=e^y\, W'(e^y)\geq0$, and $w$ is nondecreasing.  
		
		Since $0\leq\eta\leq1$, we obtain 
		\begin{equation*}
			0\leq W(x) \leq x^2\int_x^\infty \frac{\dd s}{s^3} = \frac{1}{2}.
		\end{equation*}
		We write the derivative of $w$ as follows:
		\begin{equation}\label{eq:cutoff-profile-derivative}
			w'(y) = e^y\, W'(e^y) = 2\, W(e^y)-\eta(e^y).
		\end{equation}
		Consequently, $0\leq w'\leq 1$, which proves that $w'\in \Ll^\infty(\R)$.
		
		For $0<x\leq1$, we have $J(x)=0$ and therefore $W(x)=j_{\infty}\, x^2$. For $x\geq2$, we have $\eta(s)=1$ for $s\geq x$, and hence $W(x)=\frac{1}{2}$ by \eqref{eq:large-J} and \eqref{eq:cutoff-profile}. Thus, we obtain 
		\begin{equation*}
			w(y)= \begin{cases}
				j_{\infty}\, e^{2y} & \text{ if }y\leq 0,\\[2mm]
				\displaystyle \frac{1}{2} & \text{ if }y\geq \log 2.
			\end{cases}
		\end{equation*} 
		This proves \eqref{eq:limits-w}. Since $w'\geq0$, we have
		\begin{equation*}
			\int_\R|w'(y)|\,\dd y =\lim_{R\to\infty}\int_{-R}^R w'(y)\,\dd y  =\lim_{R\to\infty}\bigl(w(R)-w(-R)\bigr)=\frac12.
		\end{equation*}
		In particular, we obtain $w'\in \Ll^1(\R)$ and \eqref{eq:integral-w}.
		
		It remains to prove \eqref{eq:scale-shift}. Let $a\geq 0$. Using $w'\geq 0$ and \eqref{eq:integral-w}, we have
		\begin{align*}
			\int_\R|w(a-x)-w(-x)|\,\dd x & = \int_{\R} \left| \int_{0}^a w'(s-x)\, \dd s\right| \, \dd x \\[2mm]
			&=\int_0^a\int_\R w'(s-x)\,\dd x\,\dd s\\[2mm]
			&=a\int_\R w'(y)\,\dd y=\frac{a}{2}.
		\end{align*}
		This is \eqref{eq:scale-shift}. If $a<0$, then we repeat the same computation with 
		\begin{equation*}
			w(a-x)-w(-x)=-\int_a^0 w'(s-x)\,\dd s.
		\end{equation*}
		We obtain $-|a|/2=a/2$. This is again \eqref{eq:scale-shift}.
	\end{proof}
	
	\subsection{Cancellation of the round part for smooth curves}
	
	Throughout this subsection, \emph{$\Gamma$ denotes an analytic Jordan curve}. In the proof of the main theorem, the results will be applied to the curves $\Gamma_j$ defined in \eqref{eq:approximation}. 
	
	To simplify the notation, we write $E_+=\Ep_{\Omega_+}\circ f$ and $E_-=\Ep_{\Omega_-}\circ g$. We write $Z_+= X_+ + \ii Y_+$ and $Z_-=X_-+\ii Y_-$ for their horizontal components. We begin by proving the following lemma.

	\begin{lemma}\label{lm:cancellation}
		We have 
		\begin{equation}\label{eq:projected-area-cancellation}
			\int_{\D}E_+^*(\dd X\wedge\dd Y) +\int_{\D^*}E_-^*(\dd X\wedge\dd Y)=0.
		\end{equation}
	\end{lemma} 
	\begin{proof} 
		Since $\Gamma$ is analytic, we can extend $f$ and $g$ across $\s^1$ by Schwarz reflection to a neighborhood of $\s^1$. In particular, $\Nc(f)$, $\Nc(g)$, and their tangential derivatives are bounded near $\s^1$. From \eqref{eq:epstein-f-coordinates} and \eqref{eq:flux-notation}, we deduce the following asymptotic expansion in the $\Cc^1$ topology with respect to $\theta$:
		\begin{equation*}
			\psi_f(re^{\ii\theta}) \underset{r\to 1}{=} re^{-\ii\theta}+O_{\Cc^1_\theta}(\rho ), \qquad \text{ and }\qquad \frac{f'(re^{\ii\theta})\, (1-r^2)\,\overline{\psi_f(re^{\ii\theta})}}{1+|\psi_f(re^{\ii\theta})|^2} \underset{r\to 1}{=} O_{\Cc^1_\theta}(\rho).
		\end{equation*}
		We deduce the following expansion for $Z_+$:
		\begin{equation}\label{eq:interior-horizontal-boundary}
			Z_+(re^{\ii\theta})  =f(re^{\ii\theta})+O_{\Cc^1_\theta}(1-r^2).
		\end{equation}
		For the map $g$, we obtain similar estimates by using the inversion $s\coloneqq \frac{1}{r}$. Since $\rho_-=\frac{1}{r^2}-1$, we obtain 
		\begin{equation*}
			\psi_g(s e^{\ii \theta})\ust{s\to 1}{=} s e^{-\ii \theta}+O_{\Cc^1_\theta}(\rho_- ),
			\qquad \text{ and }\qquad 
			\frac{g'(s e^{\ii \theta}) \rho_-(s)\, \overline{\psi_g(s e^{\ii \theta})}}{1+|\psi_g(s e^{\ii \theta})|^2} =O_{\Cc^1_\theta}(\rho_-).
		\end{equation*}
		Therefore,
		\begin{equation}\label{eq:exterior-horizontal-boundary}
			Z_-(se^{\ii\theta}) \ust{s\to 1}{=} g(se^{\ii\theta})+O_{\Cc^1_\theta}(s^2-1) .
		\end{equation}
		Since
		conformal maps preserve cyclic order, both \(f(e^{i\theta})\) and
		\(g(e^{i\theta})\), with increasing \(\theta\), traverse \(\Gamma\)
		counterclockwise.
		
		We now consider the asymptotic expansion of $Z_-$ at $\infty$. We use the Laurent expansion of $g$, with $\lambda\coloneqq g'(\infty)>0$,
		\begin{equation*}
			g(z) \ust{|z|\to \infty}{=} \lambda z+b_0+O_{\Cc^2}\left( \frac{1}{z} \right).
		\end{equation*}
		We obtain the following uniform expansions in $\theta$:
		\begin{equation*}
			\begin{cases}
				g'(s\, e^{\ii\theta}) \ust{s\to \infty}{=} \lambda+O_{\Cc^1_{\theta}}(s^{-2}), \\[2mm]
				\Nc(g)(s e^{\ii \theta}) \ust{s\to \infty}{=} O_{\Cc^1_{\theta}}(s^{-3}).
			\end{cases} 
		\end{equation*}
		It follows from \eqref{eq:def-psig} that
		\begin{equation*}
			\psi_g(s e^{\ii\theta})  \ust{s\to \infty}{=} se^{-\ii\theta}+O_{\Cc^1_\theta}(s^{-1}).
		\end{equation*}
		Hence,
		\begin{equation*}
			1+|\psi_g(s e^{\ii\theta})|^2  \ust{s\to \infty}{=} s^2+O_{\Cc^1_\theta}(1).
		\end{equation*}
		Thus, we obtain 
		\begin{equation*}
			\frac{s^2-1}{1+|\psi_g(s e^{\ii\theta})|^2}  \ust{s\to \infty}{=} 1+O_{\Cc^1_\theta}(s^{-2}).
		\end{equation*}
		Therefore, formula \eqref{eq:coordinates-Epstein} yields, uniformly in $\theta$,
		\begin{equation}\label{eq:exterior-horizontal-infinity}
			\begin{cases} 
				Z_-(s e^{\ii\theta}) \ust{s\to \infty}{=} b_0+O(s^{-1}), \\[2mm]
				\partial_\theta Z_-(se^{\ii\theta}) \ust{s\to \infty}{=} O(s^{-1}).
			\end{cases} 
		\end{equation}
		
		We now integrate by parts, using that 
		\begin{equation*}
			\dd X\wedge \dd Y = \dd\left( \frac{X\, \dd Y - Y\, \dd X}{2} \right).
		\end{equation*}
		If $\gamma\colon [0,2\pi]\to\C$ is a closed $\Cc^1$ curve, we define
		\begin{equation*}
			I(\gamma) \coloneqq \frac{1}{2} \int_0^{2\pi} \Im\left(\overline{\gamma(\theta)}\, \gamma'(\theta)\right)\,\dd\theta.
		\end{equation*}
		If $\D_r$ denotes the disk of radius $r$ centered at the origin, then
		\begin{equation*}
			\int_{\D_r} E_+^*(\dd X\wedge \dd Y)= I (\theta\mapsto Z_+(r e^{\ii \theta })).
		\end{equation*}
		Similarly, for $E_-$,
		\begin{equation*}
			\int_{\D_R\setminus \D_r} E_-^*\left(\dd X\wedge \dd Y\right) = I (\theta\mapsto Z_-(R e^{\ii \theta}))-I(\theta\mapsto Z_-(r e^{\ii \theta})).
		\end{equation*}
		By \eqref{eq:exterior-horizontal-infinity}, \eqref{eq:interior-horizontal-boundary}, and \eqref{eq:exterior-horizontal-boundary}, we have 
		\begin{align*}
			\lim_{r\to 1^-} I (\theta\mapsto Z_+(r e^{\ii \theta })) &=\frac12\int_\Gamma(X\dd Y-Y\dd X),\\[2mm]
			\lim_{r\to 1^+} I(\theta\mapsto Z_-(r e^{\ii \theta})) &=\frac{1}{2}\int_\Gamma(X\dd Y-Y\dd X), \\[2mm]
			\lim_{R\to \infty} I (\theta\mapsto Z_-(R e^{\ii \theta})) & = 0.
		\end{align*}
		We conclude by adding the three limits.
	\end{proof} 
	
	Lemma \ref{lm:cancellation} provides the following formula for the cutoff volume between the Epstein surfaces.
	
	\begin{lemma}\label{prop:exact-weighted-volume}
		For every analytic Jordan curve $\Gamma\subset\Chat$ and any $\eps>0$, we have 
		\begin{equation}\label{eq:exact-weighted-volume-flux}
			V_\eps(\Gamma) =\int_{\D} W(\xi_f/\eps)\, \sigma_f\,\dd x\,\dd y +\int_{\D^*}W(\xi_g/\eps)\, \sigma_g^-\,\dd x\, \dd y,
		\end{equation}
		where $\sigma_f$ and $\sigma_g^-$ are the signed flux densities in \eqref{eq:exact-epstein-flux} and \eqref{eq:exterior-epstein-flux}.
	\end{lemma}
	
	\begin{proof}
		All the integral computations below are first performed on the truncations \(|z|<r\) and \(1/r<|w|<R\). We then let
		\(R\to\infty\) and only afterward \(r\uparrow1\), exactly as in the proof of \eqref{eq:projected-area-cancellation}.  
		
		The function $p_{\eps}$ defined in \eqref{eq:p-epsilon} and the function $J$ defined in \eqref{eq:defJ} satisfy
		\begin{equation*}
			p_\eps(\xi) = \int_0^{\xi} \eta\left(\frac{s}{\eps}\right)\, \frac{ds}{s^3} = \int_0^{\frac{\xi}{\eps}} \eta(\sigma)\, \frac{d\sigma}{\eps^2\, \sigma^3} =\eps^{-2}J(\xi/\eps).
		\end{equation*}
		Combining \eqref{eq:p-epsilon} and \eqref{eq:volume-boundary-primitive}, we obtain
		\begin{equation*}
			\begin{aligned}
				- \eps^2\, V_{\eps}(\Gamma) & = \int_{\Omega_+} \Ep_{\Omega_+}^*\left(\eps^2 \alpha_{\eps}\right) + \int_{\Omega_-} \Ep_{\Omega_-}^*\left(\eps^2 \alpha_{\eps}\right) \\[2mm]
				& = \int_{\Omega_+} \Ep_{\Omega_+}^*\left(J(\xi/\eps)\, \dd X\wedge \dd Y\right) + \int_{\Omega_-} \Ep_{\Omega_-}^*\left(J(\xi/\eps)\, \dd X\wedge \dd Y\right).
			\end{aligned}
		\end{equation*}
		Adding $j_{\infty}$ times \eqref{eq:projected-area-cancellation}, with $j_{\infty}$ defined in \eqref{eq:jinfty}, yields
		\begin{equation*}
			\eps^2\, V_{\eps}(\Gamma) = \int_{\Omega_+} \Ep_{\Omega_+}^*\left( \left[ j_{\infty} -J(\xi/\eps)\right]\, \dd X\wedge \dd Y\right) + \int_{\Omega_-} \Ep_{\Omega_-}^*\left( \left[ j_{\infty} -J(\xi/\eps)\right] \, \dd X\wedge \dd Y\right).
		\end{equation*}
		We obtain the result by using \eqref{eq:exact-epstein-flux}, \eqref{eq:exterior-epstein-flux}, and \eqref{eq:cutoff-profile}.
	\end{proof}
	
	If $\Gamma$ is the round circle, the two integrals in \eqref{eq:exact-weighted-volume-flux} cancel exactly since the two Epstein surfaces coincide. This implies that
	\begin{equation}\label{eq:round-integral-cancellation}
		\int_1^\infty W(h_-(R)/\eps)\, \sigma_{\Id}^-(R)\, R\,\dd R =-\int_0^1 W(h_+(r)/\eps)\, \sigma_{\Id}^+(r)\, r\,\dd r,
	\end{equation}
	where $h_-$ and $h_+$ describe the heights of the Epstein surfaces given in \eqref{eq:Epstein-round}:
	\begin{equation}\label{eq:round-height-h}
		h_+(r) \coloneqq \frac{1-r^2}{1+r^2}, \qquad \text{ and }\qquad h_-(R) \coloneqq  \frac{R^2-1}{R^2+1}.
	\end{equation}
	
	\subsection{Boundary term}
	
	Using the functions in \eqref{eq:round-height-h}, we now further decompose the terms $L_f^+$, $L_g^-$, $\xi_f$, and $\xi_g$.
	
	\begin{lemma}\label{lm:decompo-Lf}
		Let $(\Gamma_j)_{j\in\N}$ be a sequence whose range is contained in a compact subset of the Weil--Petersson space.
		Let $u_{f_j}(r e^{\ii \theta}) \coloneqq \log |f_j'(re^{\ii \theta})|$ and consider the change of variable $t\coloneqq -\log h_+(r)$. Then
		\begin{equation}\label{eq:linear-flux-DtN}
			L_{f_j}^+(re^{\ii\theta})\, r\,\dd r  = \frac{2r}{1+r^2}\,\partial_ru_{f_j}(re^{\ii\theta})\,\dd t.
		\end{equation}
		Moreover, there exists a function $v_{f_j}$ such that 
		\begin{equation}\label{eq:height-factorization}
			\xi_{f_j}=  h_+(r)\, e^{u_{f_j}}\,e^{v_{f_j}},
		\end{equation}
		and the following estimate holds with a constant $C>0$ independent of $j$:
		\begin{equation}\label{eq:rho-f-bound}
			|v_{f_j}| \leq C\, \rho_+\, |\Nc(f_j)|.
		\end{equation}
	\end{lemma}
	\begin{proof} 
		The radial derivative of $u_{f_j}$ is given by
		\begin{equation}\label{eq:normal-derivative-log-f}
			\begin{aligned} 
				\partial_ru_{f_j}(re^{i\theta}) & = \frac{1}{2}\, \frac{\partial_r |f_j'(re^{i\theta})|^2}{|f_j'(re^{i\theta})|^2} \\[2mm]
				& = \frac{\Re\left( \overline{f_j'(r e^{\ii \theta})}\, \partial_r [f_j'(r e^{\ii \theta})]\right)}{|f_j'(r e^{\ii \theta})|^2} \\[2mm]
				& = \Re\left(\frac{\partial_r [f_j'(r e^{\ii \theta})]}{f_j'(r e^{\ii \theta})}\right) \\[2mm]
				& =\Re\left(e^{\ii\theta}\, \Nc(f_j)(re^{\ii\theta})\right).
			\end{aligned} 
		\end{equation}
		We have
		\begin{equation}\label{eq:dt-sigma}
			\begin{aligned}
				\dd t & \coloneqq -\dd\log h_+ \\[2mm]
				& = -\left( \frac{-2r}{1-r^2} - \frac{(1-r^2)2r(1+r^2)^{-2} }{  \frac{1-r^2}{1+r^2} } \right) \dd r \\[2mm]
				& = \left( \frac{2r}{1-r^2} + \frac{2r }{  1+r^2 } \right) \dd r \\[2mm]
				& = \sigma_{\Id}^+(re^{\ii\theta})\, r\,\dd r.
			\end{aligned}
		\end{equation}
		Combining \eqref{eq:normal-derivative-log-f} with \eqref{eq:round-height-h}, we obtain the following expression for $L_{f_j}^+$ in Lemma \ref{lm:decompo-sigmaf}:
		\begin{equation*}
			\begin{aligned} 
				L_{f_j}^+(re^{\ii\theta})\, r\,\dd r & = \frac{8\, \Re\left(r\, e^{\ii\theta}\, \Nc(f_j)(r e^{\ii \theta})\right)}{(1-r^2)(1+r^2)^2 }\, r\, \dd r \\[2mm]
				& = \frac{8\, r^2}{(1-r^2)(1+r^2)^2}\, \dr_r u_{f_j}(r e^{\ii\theta})\, \dd r\\[2mm]
				& = \frac{2r^2}{1+r^2}\, \sigma_{\Id}^+\, \dr_r u_{f_j}(r e^{\ii\theta})\, \dd r\\[2mm]
				& =\frac{2r}{1+r^2}\,\partial_ru_{f_j}(re^{\ii\theta})\,\dd t.
			\end{aligned} 
		\end{equation*}
		This proves \eqref{eq:linear-flux-DtN}. We now decompose the height $\xi_{f_j}$ given by \eqref{eq:epstein-f-coordinates}. With $v_{f_j}\coloneqq \log\frac{1+r^2}{1+|\psi_{f_j}|^2}$, we have
		\begin{equation}
			\xi_{f_j}= \frac{e^{u_{f_j}}\, (1-r^2)}{1+|\psi_{f_j}|^2} =  h_+(r)\, e^{u_{f_j}}\,e^{v_{f_j}}. 
		\end{equation}
		This proves \eqref{eq:height-factorization}. We now estimate $v_{f_j}$. From $1+|\psi_{f_j}|^2\geq 1$ and \eqref{eq:psi-square}, we obtain
		\begin{equation*}
			\left|\frac{1+r^2}{1+|\psi_{f_j}|^2}-1\right| =\left| \frac{r^2-|\psi_{f_j}|^2}{1+|\psi_{f_j}|^2} \right| \leq C\, \left( 1+\rho_+\, |\Nc(f_j)|\right) \rho_+\, |\Nc(f_j)|.
		\end{equation*}
		Using \eqref{eq:boundary-Schwarzian} and the convergence of $\Gamma_j$ in the WP topology, we find that $\rho_+|\Nc(f_j)|$ is uniformly bounded in $\Ll^{\infty}(\D)$ and is uniformly small near the boundary.
		Using the estimate $|\log(1+x)|\leq c_{\delta}|x|$ for $|x|\leq 1-\delta$ for any $\delta\in(0,1)$, we obtain near the boundary
		\begin{equation}\label{eq:bound-v-boundary}
			|v_{f_j}| \leq C\, \rho_+\, |\Nc(f_j)|.
		\end{equation}
		In the interior, $(\psi_j)_{j\in\N}$ is uniformly bounded on every compact subset of $\D$. Hence the estimate \eqref{eq:bound-v-boundary} holds on the whole disk $\D$. This proves \eqref{eq:rho-f-bound}.
	\end{proof} 
	
	For the map $g$, we have the following analogous expansion. 
	\begin{lemma}\label{lm:decompo-Lg}
		Let $(\Gamma_j)_{j\in\N}$ be a sequence whose range is contained in a compact subset of the Weil--Petersson space.
		Let $u_{g_j}(s e^{\ii \theta}) \coloneqq \log |g_j'(se^{\ii \theta})|$ and consider the change of variable $\tau \coloneqq \log h_-(s)$. Then
		\begin{equation}\label{eq:linear-flux-DtN-g}
			L_{g_j}^-(se^{\ii\theta})\, s\,\dd s  = -\frac{2s}{1+s^2}\,\partial_s u_{g_j}(se^{\ii\theta})\,\dd \tau.
		\end{equation}
		Moreover, there exists a function $v_{g_j}$ such that 
		\begin{equation} \label{eq:exterior-height-factorization}
			\xi_{g_j}=  h_-(s)\, e^{u_{g_j}}\,e^{v_{g_j}},
		\end{equation}
		and the following estimate holds for some constant $C>0$ independent of $j$:
		\begin{equation}\label{eq:rho-g-bound}
			|v_{g_j}| \leq C\, \rho_-\, |\Nc(g_j)|.
		\end{equation}
	\end{lemma}
	
	\begin{proof} 
		As in \eqref{eq:normal-derivative-log-f}, the radial derivative of $u_{g_j}$ is given by
		\begin{equation}\label{eq:normal-derivative-log-g}
			\partial_s u_{g_j}(se^{i\theta})  =\Re\left(e^{\ii\theta}\, \Nc(g_j)(se^{\ii\theta})\right).
		\end{equation}
		We have
		\begin{equation}\label{eq:dt-sigma-}
			\begin{aligned}
				\dd \tau & \coloneqq \dd\log h_- \\[2mm]
				& = \left( \frac{2s}{s^2-1} - \frac{(s^2-1)2s(1+s^2)^{-2} }{  \frac{s^2-1}{s^2+1} } \right) \dd s \\[2mm]
				& = \left( \frac{2s}{s^2-1} - \frac{2s }{  s^2+1 } \right) \dd s \\[2mm]
				& = -\sigma_{\Id}^-(se^{\ii\theta})\, s\,\dd s.
			\end{aligned}
		\end{equation}
		From \eqref{eq:normal-derivative-log-g}, we obtain the following expression for $L_{g_j}^-$ in Lemma \ref{prop:exterior-epstein-flux}:
		\begin{equation*}
			\begin{aligned} 
				L_{g_j}^-(se^{\ii\theta})\, s\,\dd s & = -\frac{8\, \Re\left(s\, e^{\ii\theta}\, \Nc(g_j)(s e^{\ii \theta})\right)}{(s^2-1)(s^2+1)^2 }\, s\, \dd s \\[2mm]
				& = \frac{-8\, s^2}{(s^2-1)(s^2+1)^2}\, \dr_s u_{g_j}(s e^{\ii\theta})\, \dd s\\[2mm]
				& = \frac{2s^2}{s^2+1}\, \sigma_{\Id}^-\, \dr_s u_{g_j}(s e^{\ii\theta})\, \dd s\\[2mm]
				& = -\frac{2s}{1+s^2}\,\partial_su_{g_j}(se^{\ii\theta})\,\dd \tau.
			\end{aligned} 
		\end{equation*}
		This proves \eqref{eq:linear-flux-DtN-g}. The proofs of \eqref{eq:exterior-height-factorization} and \eqref{eq:rho-g-bound} are identical to those of \eqref{eq:height-factorization} and \eqref{eq:rho-f-bound}.
	\end{proof} 
	
	We also use the following localization, which relates Epstein height to
	the radial coordinate.
	
	\begin{lemma}\label{lem:small-height-localization}
		Let $(\Gamma_j)_{j\in\N}$ be a sequence whose range is contained in a compact subset of the Weil--Petersson space. For every $1>\delta>0$, there exists $c_\delta>0$ such that, for every $j\in\N$, the two Epstein heights satisfy
		\begin{equation}\label{eq:height-positive-on-compacta}
			\xi_{f_j}\geq c_\delta\quad \text{ in }\D_{1-\delta},
			\qquad \text{ and }\qquad 
			\xi_{g_j}\geq c_\delta\quad \text{ in }\C\setminus \D_{1+\delta}.
		\end{equation}
		Hence, if $2\eps<c_\delta$, then, for all $j\geq 0$, we have $\{\xi_{f_j}<2\eps\}\subset \D\setminus \D_{1-\delta}$ and $\{\xi_{g_j}<2\eps\}\subset \D_{1+\delta}\setminus\D$.
	\end{lemma}
	
	\begin{proof}
		This follows from Corollary 3.13 and the proof of Lemma 5.17 in \cite{BBVPW2025}.
	\end{proof}
	
	\subsection{Poisson extension and the cutoff terms}

	For $x,a\in\R$, we use \eqref{eq:cutoff-profile} to define
	\begin{equation}\label{eq:Phi-definition}
		\Phi_x(a) \coloneqq w(a-x)-w(-x) =W(e^{a-x})-W(e^{-x}).
	\end{equation}
	Since \(w\) is nondecreasing by Lemma \ref{lem:cutoff-function-properties}, we have $a\,\Phi_x(a)\geq0$. We rewrite \eqref{eq:integral-w} and \eqref{eq:scale-shift} as follows:
	\begin{align}
		\int_{\R}\Phi_x(a)\,\dd x&=\frac{a}{2}, \label{eq:Phi-shift-integral}\\
		\int_{\R}|\Phi_x(a)-\Phi_x(b)|\,\dd x &=\frac{1}{2}|a-b|. \label{eq:Phi-shift-L1}
	\end{align}
	
	We now use the Dirichlet-to-Neumann operator $\Lambda$ and the bilinear form $\Er$ defined in Section \ref{sec:Poisson}.
	
	\begin{lemma}
		\label{lem:Phi-identities}
		For every $u\in \W^{\frac{1}{2},2}(\s^1,\R)$, we have
		\begin{equation}\label{eq:Phi-zero-order}
			\int_\R\int_{\s^1}\Phi_x(u)\,\dd\theta\, \dd x =\frac{1}{2} \int_{\s^1} u\,\dd\theta.
		\end{equation}
		We have
		\begin{equation}\label{eq:Phi-energy-density}
			d_u(x) \coloneqq \Er(u,\Phi_x(u))\geq 0.
		\end{equation}
		We have
		\begin{equation}\label{eq:Phi-energy-integral}
			\int_{\R} d_u(x)\,\dd x =\frac{1}{2}\, \Er(u,u). 
		\end{equation}
		Moreover, the maps $u\in \W^{\frac{1}{2},2}(\s^1,\R)\mapsto \Phi_{\cdot}(u)\in \Ll^1(\R\times \s^1)$ and $u\in  \W^{\frac{1}{2},2}(\s^1,\R)\mapsto d_u\in \Ll^1(\R)$ are continuous. For any $x\in\R$, the map $u\in\W^{\frac{1}{2},2}(\s^1,\R)\mapsto \Phi_x(u)\in \W^{\frac{1}{2},2}(\s^1,\R)$ is also continuous. 
	\end{lemma}
	
	\begin{proof}
		By \eqref{eq:Phi-shift-integral}, we have
		\begin{equation*}
			\forall \theta\in[0,2\pi],\qquad \int_{\R}\Phi_x(u(e^{\ii\theta}))\,\dd x=\frac{1}{2}\, u(e^{\ii\theta}).
		\end{equation*}
		We obtain \eqref{eq:Phi-zero-order} by Fubini's theorem.

		By the definition of $\Phi$ in \eqref{eq:Phi-definition}, we have
		\begin{equation*}
			|\Phi_x(a)-\Phi_x(b)| \leq\|w'\|_{L^\infty}|a-b|.
		\end{equation*}
		Hence, $\Phi_x(u)\in \W^{\frac{1}{2},2}(\s^1)$ and the pairing in \eqref{eq:Phi-energy-density} is well defined. For any $x\in\R$, since $\Phi_x$ is globally Lipschitz, the map $u\in\W^{\frac{1}{2},2}(\s^1,\R)\mapsto \Phi_x(u)\in \W^{\frac{1}{2},2}(\s^1,\R)$ is also continuous.
		
		The function \(a\mapsto\Phi_x(a)\) is nondecreasing. Hence, the integrand in the following formula is pointwise nonnegative:
		\begin{align*}
			d_u(x)=\frac{1}{2\pi}\int_{0}^{2\pi}\int_{0}^{2\pi} \frac{(u(e^{\ii\theta})-u(e^{\ii\vp})) (\Phi_x(u(e^{\ii\theta}))-\Phi_x(u(e^{\ii\vp})))}{|e^{\ii\theta}-e^{\ii\varphi}|^2} \,\dd\theta\, \dd\varphi \geq 0.
		\end{align*}
		Furthermore, by \eqref{eq:Phi-shift-L1}, we have
		\begin{equation*}
			\int_{\R} \left|(u(e^{\ii\theta})-u(e^{\ii\vp}))\, (\Phi_x(u(e^{\ii\theta}))-\Phi_x(u(e^{\ii\vp})))\right|\,\dd x  =\frac{1}{2}\, |u(\theta)-u(\varphi)|^2.
		\end{equation*}
		Hence, we obtain \eqref{eq:Phi-energy-integral} by Fubini's theorem:
		\begin{align*}
			\int_\R d_u(x)\,\dd x &=\frac{1}{2\pi}\int_{\s^1} \int_{\s^1} \frac{u(e^{\ii\theta}) -u( e^{\ii \vp} ) }{ |e^{i\theta}-e^{i\vp}|^2}
			\left(\int_{\R} [\Phi_x(u(e^{\ii\theta}))-\Phi_x(u(e^{\ii\vp}))]\,\dd x\right)\,\dd\theta\, \dd\vp \\[2mm]
			&=\frac{1}{4\pi}\int_{\s^1} \int_{\s^1} \frac{|u(e^{\ii\theta})-u(e^{\ii\vp})|^2}{|e^{\ii\theta}-e^{\ii\vp}|^2}\,\dd\theta\, \dd\vp \\[2mm]
			&=\frac{1}{2}\, \Er(u,u).
		\end{align*}
		The continuity of the map $u\in \W^{\frac{1}{2},2}(\s^1,\R)\mapsto \Phi_x(u)\in \Ll^1(\R\times \s^1)$ follows directly from \eqref{eq:Phi-shift-L1}. To prove the continuity of the map $u\mapsto d_u$, we consider a sequence $(u_j)_{j\in\N}\subset \W^{\frac{1}{2},2}(\s^1,\R)$ converging strongly to some $u\in \W^{\frac{1}{2},2}(\s^1,\R)$. Every subsequence has a further subsequence that converges almost everywhere to $u$. By the continuity of the map $(u\mapsto\Phi_x(u))$ in $\W^{\frac{1}{2},2}$ and Theorem \ref{th:Brezis-Lieb}, every subsequence of $(d_{u_j})_{j\in\N}$ also has a subsequence that converges almost everywhere to $d_u$. Moreover, the $\Ll^1$ norms of $d_{u_j}$ are exactly $\frac{1}{2}\Er(u_j,u_j)$, which converge to $\frac{1}{2}\Er(u,u)$ by strong convergence in $\W^{\frac{1}{2},2}(\s^1)$. This is $\|d_u\|_{\Ll^1(\R)}$. By Theorem \ref{th:Brezis-Lieb}, we obtain $d_{u_j}\to d_u$ in $\Ll^1(\R)$. Hence, $u\mapsto d_u$ is continuous.
	\end{proof}
	
	We now prove the monotonicity formula.
	
	\begin{lemma}\label{lem:Poisson-semigroup-monotonicity}
		Let $F\in \Cc^1(\R)$ satisfy $0\leq F'\leq L$. Then, for every $u\in \W^{\frac{1}{2},2}(\s^1;\R)$ and $0<r<1$, we have
		\begin{equation}\label{eq:Poisson-monotonicity}
			0\leq \Er(\Pcal_r u,F(\Pcal_r u)) \leq \Er(u,F(u)).
		\end{equation}
	\end{lemma}
	
	\begin{proof}
		We first assume that $u$ is $\Cc^{\infty}$. For $s\geq0$, we denote $v(s,\theta)\coloneqq \Pcal_{e^{-s}}u(\theta)$.
		By Section \ref{sec:Poisson}, we have 
		\begin{equation}\label{eq:syst-v}
			\begin{cases} 
				\dr_s v =-\Lambda v, \\[2mm]
				\dr_{ss}^2v+\dr_{\theta\theta}^2 v=0.
			\end{cases}
		\end{equation} 	
		Let $G$ be a primitive of $F$, and define
		\begin{equation*}
			J(s) \coloneqq \int_{\s^1} G(v(s,\theta))\,\dd\theta.
		\end{equation*}
		Since $\partial_sv=-\Lambda v$, we obtain from the first equation of \eqref{eq:syst-v}
		\begin{equation}\label{eq:circular-mean-first-derivative}
			J'(s) =\int_{\s^1} F(v)\, \partial_sv\,\dd\theta =- \int_{\s^1} F(v)\, \Lambda v\, \dd\theta =-\Er(v,F(v)).
		\end{equation}
		Using $\dr^2_{ss} v=-\dr^2_{\theta\theta} v$ and integrating by parts in $\theta$, we also obtain
		\begin{equation}\label{eq:circular-mean-second-derivative}
			J''(s) =\int_{\s^1} \left[ F'(v)|\dr_s v|^2 +F(v)\dr_{ss}^2 v \right]\dd\theta  
			=\int_{\s^1} F'(v) \left( |\partial_sv|^2+|\partial_\theta v|^2\right)\dd\theta
			\geq0.
		\end{equation}
		Hence $J$ is convex, and the map $s\mapsto -J'(s) = \Er(\Pcal_{e^{-s}}u,F(\Pcal_{e^{-s}}u))$ is nonincreasing.
		
		Moreover, we have $-J'\geq 0$. Indeed, we have 
		\begin{equation*}
			\Er(v,F(v)) =\frac{1}{2\pi}\int_{\s^1}\int_{\s^1} \frac{ (v(\theta)-v(\vp)) (F(v(\theta))-F(v(\vp))) }{ |e^{\ii\theta}-e^{\ii\vp}|^2 } \,\dd\theta\, \dd\vp.
		\end{equation*}
		The integrand is pointwise nonnegative because $F$ is nondecreasing. Taking $s=-\log r$, we obtain \eqref{eq:Poisson-monotonicity} for $u$ of class $\Cc^{\infty}$.
		
		For $u\in \W^{\frac{1}{2},2}(\s^1)$, we let $u_\rho\coloneqq \Pcal_\rho u$ for $0<\rho<1$. 
		Then each $u_\rho$ is smooth, with $u_{\rho}\to u$ and $\Pcal_r u_{\rho}\to \Pcal_r u$ in $\W^{\frac{1}{2},2}(\s^1)$ as $\rho\to 1$. This follows from the Fourier-series formulas. Hence, we can apply \eqref{eq:Poisson-monotonicity} to $u_{\rho}$ and let $\rho\to 1$.
	\end{proof}
	
	Given $t\geq 0$, we define $r(t)\geq 0$ such that $h_+(r(t))=e^{-t}$, where $h$ is defined in \eqref{eq:round-height-h}. Then we have the formula 
	\begin{equation}\label{eq:r-of-t}
		r(t) = \left( \frac{1-e^{-t}}{1+e^{-t}} \right)^{\frac{1}{2}}.
	\end{equation}
	For $T>0$ and $x\geq-T$, we denote
	\begin{equation}\label{eq:Poisson-radius}
		r_{T,x} \coloneqq r(T+x).
	\end{equation}
	Then, for every fixed $x$,
	\begin{equation}\label{eq:limit-r-c}
		r_{T,x}\xrightarrow[T\to \infty]{} 1.
	\end{equation}
	
	\begin{proposition}\label{prop:uniform-Poisson-regularization}
		Let $K$ be a compact subset of $\W^{\frac{1}{2},2}(\s^1,\R)$. For every $T>0$, we consider a bounded measurable function $s_{T,\cdot}\colon [-T,\infty) \to [0,+\infty)$ such that there exists $C>0$ for which
		\begin{equation}\label{eq:hyp-a}
			\begin{cases}
				\forall T\geq 0,\ \forall x\geq -T,\qquad 0\leq s_{T,x}\leq C, \\[2mm]
				\forall x\in\R,\qquad s_{T,x} \xrightarrow[T\to +\infty]{} 1.
			\end{cases}
		\end{equation}
		Then
		\begin{equation}\label{eq:uniform-Poisson-zero-order}
			\lim_{T\to\infty}\sup_{u\in K} \left| \int_{-T}^{\infty}\int_{\s^1} \Phi_x(\Pcal_{r_{T,x}}u)\,\dd\theta\, \dd x -\frac{1}{2}\int_{\s^1}u\,\dd\theta \right| =0,
		\end{equation}
		and 
		\begin{equation}\label{eq:uniform-Poisson-energy}
			\lim_{T\to\infty}\sup_{u\in K} \left| \int_{-T}^{\infty}s_{T,x}\, \Er\left( \Pcal_{r_{T,x}} u, \Phi_x( \Pcal_{r_{T,x}} u) \right)\,\dd x -\frac{1}{2}\Er(u,u) \right| =0.
		\end{equation}
	\end{proposition}
	
	\begin{proof}
		It is enough to consider arbitrary sequences $T_j\to\infty$ and $u_j\in  K$. By compactness of $K$, after passing to a subsequence, we may assume that $u_j\to u\in K$ strongly in $\W^{\frac{1}{2},2}(\s^1)$.
		
		Since $\Pcal_r u\to u$ in $\W^{\frac{1}{2},2}(\s^1)$ as $r\to 1^-$, using \eqref{eq:limit-r-c}, we have, for every fixed $x\in\R$,
		\begin{equation}\label{eq:moving-Poisson-Hhalf-convergence}
			\begin{split}
				\|\Pcal_{r_{T_j,x}}u_j-u\|_{\W^{\frac{1}{2},2}} &\leq \|\Pcal_{r_{T_j,x}}(u_j-u)\|_{\W^{\frac{1}{2},2}} +\|\Pcal_{r_{T_j,x}}u-u\|_{\W^{\frac{1}{2},2}} \\[2mm]
				&\leq \|u_j-u\|_{\W^{\frac{1}{2},2}} +\|\Pcal_{r_{T_j,x}}u-u\|_{\W^{\frac{1}{2},2}} \xrightarrow[j\to \infty]{} 0.
			\end{split}
		\end{equation}
		We first prove \eqref{eq:uniform-Poisson-zero-order}. Define the radial
		Poisson maximal function by
		\begin{equation*}
			\Pcal_*v(\theta) = \sup_{0<r<1}|\Pcal_r v(\theta)|.
		\end{equation*}
		Using the standard comparison of the Poisson maximal function with the Hardy--Littlewood maximal function, together with the boundedness of the maximal function in $\Ll^2$ (see, for instance, \cite[Theorems 2.1.6 and 2.1.10, Example 2.1.13, pp.~88--93, Exercise~3.1.7, p.~183]{Grafakos2008}), we obtain 
		\begin{equation}
			\label{eq:Poisson-maximal-L2}
			\|\Pcal_*v\|_{\Ll^2(\s^1)} \leq C\|v\|_{\Ll^2(\s^1)}.
		\end{equation}
		For $A\geq0$, we denote
		\begin{equation*}
			H_A(x)\coloneqq w(A-x)-w(-A-x) = \Phi_x(A) - \Phi_x(-A).
		\end{equation*}
		If $|a|\leq A$, we deduce from the monotonicity of $w$ in Lemma \ref{lem:cutoff-function-properties} that
		\begin{equation}\label{eq:bound-Phi-H}
			|\Phi_x(a)| = |w(a-x)-w(-x)| \leq H_A(x).
		\end{equation}
		Moreover, we obtain from \eqref{eq:scale-shift} that 
		\begin{equation}\label{eq:Phi-maximal-envelope}
			\int_{\R} H_A(x)\,\dd x=A.
		\end{equation}
		We denote $A_j\coloneqq \Pcal_* u_j$ and $A\coloneqq \Pcal_* u$.
		By \eqref{eq:Poisson-maximal-L2}, we have\footnote{By the triangle inequality, we have $A_j \leq \sup_r \left[ |\Pcal_r(u)| + |\Pcal_r(u_j-u)| \right] \leq A + \Pcal_*(u_j-u)$. Conversely, we also have $A\leq A_j + \Pcal_*(u-u_j)$.}
		\begin{equation}\label{eq:Aj-to-A-L2}
			\|A_j-A\|_{\Ll^2(\s^1)}  \leq \|\Pcal_*(u_j-u)\|_{\Ll^2(\s^1)} \leq C\|u_j-u\|_{\Ll^2(\s^1)} \xrightarrow[j\to \infty]{} 0.
		\end{equation}
		After passing to a further subsequence, we may assume that
		\(A_j\to A\) almost everywhere on $\s^1$. Hence, we have $H_{A_j(\theta)}(x)\to H_{A(\theta)}(x)$ as $j\to \infty$ for a.e.\ $(x,\theta)\in\R\times\s^1$. On the other hand, we obtain from \eqref{eq:Phi-maximal-envelope} and \eqref{eq:Aj-to-A-L2} that 
		\begin{equation*}
			\int_{\s^1}\int_{\R} H_{A_j(\theta)}(x)\,\dd x\, \dd\theta = \|A_j\|_{L^1(\s^1)} \xrightarrow[j\to \infty]{} \|A\|_{L^1(\s^1)}.
		\end{equation*}
		We deduce from Theorem \ref{th:Brezis-Lieb} that
		\begin{equation}\label{eq:Poisson-envelope-L1-convergence}
			H_{A_j(\theta)}(x) \xrightarrow[j\to \infty]{\Ll^1(\R\times\s^1)} H_{A(\theta)}(x).
		\end{equation}
		We denote 
		\begin{equation*}
			\forall x\in\R,\ \forall \theta\in\s^1,\qquad F_j(x,\theta) \coloneqq  \mathbf{1}_{[-T_j,\infty)}(x)\,  \Phi_x(\Pcal_{r_{T_j,x}}u_j(\theta)).
		\end{equation*}
		Let $M>0$. We obtain from \eqref{eq:limit-r-c} and the monotonicity of the map $(x\mapsto r_{T,x})$ that 
		\begin{equation}\label{eq:Poisson-local-uniform-L2}
			\sup_{|x|\leq M} \|\Pcal_{r_{T_j,x}}u_j-u\|_{\Ll^2(\s^1)} \leq \|u_j-u\|_{\Ll^2(\s^1)} + \|\Pcal_{r_{T_j,-M}}u-u\|_{\Ll^2(\s^1)} \xrightarrow[j\to\infty]{} 0.
		\end{equation}
		Because $w$ is Lipschitz, we find that $F_j\to \Phi_x(u(\theta))$ in $\Ll^1([-M,M]\times \s^1)$. From \eqref{eq:bound-Phi-H}, we deduce that $|F_j(x,\theta)| \leq H_{A_j(\theta)}(x)$.
		Using \eqref{eq:Poisson-envelope-L1-convergence}, the dominated convergence theorem, and the fact that the tails $\|F_j\|_{L^1((\R\setminus [-M,M])\times \s^1)}$ converge uniformly to $0$, first letting $j\to\infty$ and then $M\to\infty$ gives
		\begin{equation*}
			F_j \xrightarrow[j\to \infty]{\Ll^1(\R\times\T)} \Phi_x(u(\theta)) .
		\end{equation*}
		Combining this with \eqref{eq:Phi-zero-order}, we obtain \eqref{eq:uniform-Poisson-zero-order} along the chosen subsequence. Since the sequences $(T_j)_{j\in\N}$ and $(u_j)_{j\in\N}$ are arbitrary, we obtain \eqref{eq:uniform-Poisson-zero-order}.
		
		We now prove \eqref{eq:uniform-Poisson-energy}. We denote
		\begin{equation*}
			G_j(x) \coloneqq \mathbf{1}_{[-T_j,\infty)}(x)\, s_{T_j,x}\,  \Er \left( \Pcal_{r_{T_j,x}}u_j, \Phi_x(\Pcal_{r_{T_j,x}}u_j) \right).
		\end{equation*}
		Applying Lemma \ref{lem:Poisson-semigroup-monotonicity} with $F=\Phi_x$ (which is possible by Lemma \ref{lem:cutoff-function-properties}, since $\dr_a \Phi_x(a) = w'(a-x)\in [0,1]$) and using \eqref{eq:hyp-a}, we obtain
		\begin{equation}\label{eq:Poisson-density-domination}
			0\leq G_j(x)\leq C\, \Er\left(u_j,\Phi_x(u_j)\right) = C\, d_{u_j}(x).
		\end{equation}
		Let $x\in\R$. By \eqref{eq:moving-Poisson-Hhalf-convergence} and the fact that $\Phi_x(\cdot)$ is globally Lipschitz, we have
		\begin{equation*}
			\Phi_x(\Pcal_{r_{T_j,x}}u_j) \xrightarrow[j\to \infty]{\W^{\frac{1}{2},2}(\s^1)} \Phi_x(u).
		\end{equation*}
		Since $s_{T_j,x}\to1$ as $j\to\infty$ by \eqref{eq:hyp-a} and $\Er$ is continuous on $\W^{\frac{1}{2},2}(\s^1)$, it follows that $G_j(x) \to \Er(u,\Phi_x(u))=d_u(x)$ a.e.
		
		By Lemma \ref{lem:Phi-identities}, we have $d_{u_j}\to d_u$ in $\Ll^1(\R)$. Thus, the family $(d_{u_j})_{j\in\N}$ is uniformly integrable. Moreover, for any $M>0$, we deduce from \eqref{eq:Poisson-density-domination} 
		\begin{equation}\label{eq:tail-G}
			\lim_{M\to\infty}\limsup_{j\to\infty} \|G_j-d_{u_j}\|_{\Ll^1(\R\setminus[-M,M])} \leq \lim_{M\to \infty} \limsup_{j\to\infty} C\|d_{u_j}\|_{\Ll^1(\R\setminus[-M,M])} =0.
		\end{equation}
		Together with \eqref{eq:Poisson-density-domination} and the a.e.\ convergence, this allows us to apply Vitali's theorem and obtain $G_j\to d_u$ in $\Ll^1([-M,M])$ for any $M>0$. Thus, for any $M>0$,
		\begin{equation*}
			\limsup_{j\to\infty}\|G_j-d_u\|_{\Ll^1(\R)} = \limsup_{j\to\infty}\|G_j-d_u\|_{\Ll^1(\R\setminus [-M,M])}.
		\end{equation*}
		We obtain from \eqref{eq:tail-G} that 
		\begin{equation}\label{eq:Poisson-energy-L1-convergence}
			G_j \xrightarrow[j\to \infty]{\Ll^1(\R)} d_u.
		\end{equation}
		Integrating and combining the result with \eqref{eq:Phi-energy-integral} gives \eqref{eq:uniform-Poisson-energy} along the chosen subsequence. Since the sequences $(T_j)_{j\in\N}$ and $(u_j)_{j\in\N}$ are arbitrary, we obtain \eqref{eq:uniform-Poisson-energy}.
	\end{proof}
	
	We now restore the exact Epstein height. The functions $u_f$ and $u_g$ are defined in Lemma \ref{lm:decompo-Lf} and Lemma \ref{lm:decompo-Lg}, respectively. The functions $u_{\pm}$ are defined above \eqref{eq:trace-energies}. In particular, we have $u_f(re^{\ii\theta})=\Pcal_r u_+(\theta)$ and $\xi_f=h(r)e^{u_f+v_f}$.
	
	\begin{lemma}\label{lem:uniform-height-correction}
		Let $(\Gamma_j)_{j\in\N}$ be a sequence whose range is contained in a compact subset of the Weil--Petersson space. Then
		\begin{equation}\label{eq:interior-height-replacement}
			\lim_{\eps \to 0^+} \sup_{j\in\N} \left| \int_{\D} \left[ W(\xi_{f_j}/\eps) -W(h_+ e^{u_{f_j}}/\eps) \right] (\sigma_{\Id}^++L_{f_j}^+)\,\dd x\, \dd y \right| =0.
		\end{equation}
		Similarly, for the exterior part,
		\begin{equation}\label{eq:exterior-height-replacement}
			\lim_{\eps \to 0^+} \sup_{j\in\N} \left| \int_{\D^*} \left[ W(\xi_{g_j}/\eps) -W(h_- e^{u_{g_j}}/\eps) \right] (\sigma_{\Id}^-+L^-_{g_j})\,\dd x\, \dd y \right| =0.
		\end{equation}
	\end{lemma}
	
	\begin{proof}
		Since the family $(\Gamma_j)_{j\in\N}$ is compact in the Weil--Petersson topology,
		\begin{equation*}
			\lim_{\delta \to 0^+} \sup_{j\in\N} \sup_{1-\delta<|z|<1} (1-|z|^2)\, |\Nc(f_j)(z)| =0,
		\end{equation*}
		and 
		\begin{equation*}
			\lim_{\delta \to 0^+} \sup_{j\in\N} \sup_{1<|w|<1+\delta} (|w|^2-1)\, |\Nc(g_j)(w)| =0.
		\end{equation*}
		Consequently, there exist $\delta_0>0$ and $C>0$ such that, by \eqref{eq:rho-g-bound} and \eqref{eq:rho-f-bound}, for all $j\in\N$ we have
		\begin{align}
			|v_{f_j}(z)| &\leq C(1-|z|^2)|\Nc(f_j)(z)|, &&\text{ if } 1-\delta_0<|z|<1, \label{eq:uniform-rho-f-bound} \\[2mm]
			|v_{g_j}(z)| &\leq C(|z|^2-1)|\Nc(g_j)(z)|, &&\text{ if }1<|z|<(1-\delta_0)^{-1}. \label{eq:uniform-rho-g-bound}
		\end{align}
		Let $0<\delta<\delta_0$. By Lemma \ref{lem:small-height-localization}, the exact heights $\xi_{f_j}$ and $\xi_{g_j}$ have a common positive lower bound away from the annuli
		\begin{equation*}
			A_\delta^+ \coloneqq \{1-\delta<|z|<1\}, \qquad A_\delta^- \coloneqq \{1<|z|<1+\delta\}.
		\end{equation*}
		On the same regions, $v_{f_j}$ and $v_{g_j}$ are uniformly bounded. Since $h_+ e^{u_{f_j}}=\xi_{f_j}\, e^{-v_{f_j}}$ and $h_-\, e^{u_{g_j}}=\xi_{g_j} e^{-v_{g_j}}$, there is a constant $c_\delta>0$ such that, for all $j\in\N$,
		\begin{equation*}
			\begin{cases} 
				\min\{\xi_{f_j},h_+ e^{u_{f_j}}\} \geq c_\delta &\text{on }\{|z|\leq1-\delta\},\\[2mm]
				\min\{\xi_{g_j},h_-\, e^{u_{g_j}}\} \geq c_\delta & \text{on }\{|z|\geq 1+\delta\}.
			\end{cases} 
		\end{equation*}
		Since $W(s)=\frac{1}{2}$ for $s\geq 2$ by \eqref{eq:large-J} and \eqref{eq:cutoff-profile}, if $2\eps<c_\delta$, then 
		\begin{equation*}
			\begin{cases} 
				W(\xi_{f_j}/\eps) = W(h_+\, e^{u_{f_j}}/\eps) & \text{on }\{|z|\leq1-\delta\}, \\[2mm]
				W(\xi_{g_j}/\eps) = W(h_-\, e^{u_{g_j}}/\eps) &\text{on }\{|z|\geq 1+\delta\}.
			\end{cases} 
		\end{equation*} 
		Thus, for $2\eps<c_{\delta}$, the differences in \eqref{eq:interior-height-replacement} and \eqref{eq:exterior-height-replacement} are supported in $A_\delta^+$ and $A_\delta^-$, respectively.
		
		Since $w(y)=W(e^y)$ and $w'\in \Ll^\infty(\R)$ by Lemma \ref{lem:cutoff-function-properties}, we obtain from the mean-value theorem and \eqref{eq:height-factorization} that 
		\begin{equation}\label{eq:interior-cutoff-height-error}
			\left| W(\xi_{f_j}/\eps) -W(h_+ e^{u_{f_j}}/\eps) \right|  \leq \|w'\|_{\Ll^{\infty}(\R)}\, \left| \log\left(\frac{\xi_{f_j}}{\eps}\right) - \log\left(\frac{h_+}{\eps}\right) - u_{f_j} \right|\leq \|w'\|_{L^\infty}|v_{f_j}|.
		\end{equation}
		Similarly, for $g_j$, we obtain from \eqref{eq:exterior-height-factorization} that
		\begin{equation}\label{eq:exterior-cutoff-height-error}
			\left| W(\xi_{g_j}/\eps) -W(h_-\, e^{u_{g_j}}/\eps) \right| \leq \|w'\|_{\Ll^\infty(\R)}\, |v_{g_j}|. 
		\end{equation}
		We first focus on \eqref{eq:interior-height-replacement}. We deduce from \eqref{eq:sigma-id-interior} and \eqref{eq:uniform-rho-f-bound} that, for all $j\in\N$,
		\begin{equation*}
			|v_{f_j}|\,|\sigma_{\Id}^+| \leq C\, \frac{|\Nc(f_j)|}{1+|z|^2} \leq C\, |\Nc(f_j)|.
		\end{equation*}
		Similarly, using Lemma \ref{lm:decompo-sigmaf}, we have, for all $j\in\N$,
		\begin{equation*}
			|v_{f_j}|\,|L^+_{f_j}| \leq C|\Nc(f_j)|^2.
		\end{equation*}
		It follows from \eqref{eq:interior-cutoff-height-error} that
		\begin{equation}\label{eq:interior-height-error-bound}
			\begin{aligned} 
				&\left| \int_{\D} \left[ W(\xi_{f_j}/\eps) -W(h_+ e^{u_{f_j}}/\eps) \right] (\sigma_{\Id}^+ +L^+_{f_j})\, \dd^2 z \right| \\[2mm]
				&\leq C\int_{A_\delta^+} \left(|\Nc(f_j)|+|\Nc(f_j)|^2\right)\, \dd^2 z \\[2mm]
				&\leq C|A^+_\delta|^{\frac{1}{2}} \left( \int_{A_\delta}|\Nc(f_j)|^2\, \dd^2 z \right)^{1/2} + C\int_{A_\delta^+}|\Nc(f_j)|^2\, \dd^2 z.
			\end{aligned} 
		\end{equation}
		Since $|A_\delta^+|\leq C\delta$, the right-hand side tends to zero as $\delta\to 0$ uniformly in $j$ by \eqref{eq:WP-convergence-data}. This proves \eqref{eq:interior-height-replacement}.
		
		The proof of \eqref{eq:exterior-height-replacement} is similar, using instead \eqref{eq:uniform-rho-g-bound} and Lemma \ref{prop:exterior-epstein-flux}.
	\end{proof}
	
	We define
	\begin{align*}
		B_{\eps,+}(f) &\coloneqq \int_\D \left[ W(\xi_f/\eps)(\sigma_{\Id}^++L^+_f) -W(h_+/\eps)\sigma_{\Id}^+ \right]\, \dd x\, \dd y,\\[2mm]
		B_{\eps,-}(g) &\coloneqq \int_{\D^*} \left[ W(\xi_g/\eps)(\sigma_{\Id}^-+L^-_g) -W(h_-/\eps)\sigma_{\Id}^- \right]\, \dd x\, \dd y, \\[2mm]
		B_\eps(f,g) & \coloneqq B_{\eps,+}(f)+B_{\eps,-}(g).
	\end{align*}
	By \eqref{eq:round-integral-cancellation}, we have 
	\begin{equation}\label{eq:Beps}
		B_{\eps}(f,g) = \int_\D  W(\xi_f/\eps)(\sigma_{\Id}^++L^+_f)\, \dd x\, \dd y + \int_{\D^*} W(\xi_g/\eps)(\sigma_{\Id}^-+L^-_g)\, \dd x\, \dd y.
	\end{equation}
	
	\begin{lemma}\label{thm:uniform-cutoff-limit}
		Let $(\Gamma_j)_{j\in\N}$ be a sequence whose range is contained in a compact subset of the Weil--Petersson space.
		Let $u_{j,+}$ and $u_{j,-}$ be the boundary traces of $\log|f_j'|$ and $\log|g_j'|$. Then
		\begin{equation}\label{eq:uniform-cutoff-limit}
			\lim_{\eps\to 0^+}\sup_{j\in\N} \left| B_\eps(f_j,g_j) -\frac{1}{2}\int_{\s^1}(u_{j,+}-u_{j,-})\,\dd\theta -\frac{1}{2}\left( \Er(u_{j,+},u_{j,+}) +\Er(u_{j,-},u_{j,-}) \right) \right| =0.
		\end{equation}
	\end{lemma}
	
	\begin{proof}
		We denote
		\begin{align*}
			B_{\eps,+}^0(f_j) &\coloneqq \int_{\D} \left[ W(h_+ e^{u_{f_j}}/\eps)(\sigma_{\Id}^++L^+_{f_j}) -W(h_+/\eps)\sigma_{\Id}^+ \right]\, \dd x\, \dd y,\\[2mm]
			B_{\eps,-}^0(g_j) &\coloneqq \int_{\D^*} \left[ W(h_- e^{u_{g_j}}/\eps)(\sigma_{\Id}^- + L^-_{g_j}) -W(h_-/\eps)\sigma_{\Id}^- \right]\, \dd x\, \dd y.
		\end{align*}
		We rewrite these two terms using the Poisson semigroup.
		We first focus on $B_{\eps,+}^0$. Given $r\in(0,1)$, we denote
		\begin{equation*}
			T=\log\frac{1}{\eps}, \qquad t=-\log h_+(r), \qquad a= t-T = \log\left(\frac{\eps}{h_+(r)}\right).
		\end{equation*}
		We then have $\frac{h_+(r)}{\eps}=e^{-a}$ and $r=r_{T,a}$ in the notation of \eqref{eq:Poisson-radius}.
		Consequently, we deduce from the definition of $\Phi$ in \eqref{eq:Phi-definition} that 
		\begin{equation}\label{eq:W-Phi}
			W(h_+ e^{u_{f_j}}/\eps)-W(h_+/\eps) = W(e^{u_{f_j}-a}) - W(e^{-a}) = \Phi_a(u_{f_j}) = \Phi_a(\Pcal_{r_{T,a}} u_{j,+}).
		\end{equation}
		Using \eqref{eq:dt-sigma} and \eqref{eq:linear-flux-DtN}, we obtain 
		\begin{equation}\label{eq:sigma-term-B0+}
			\begin{aligned} 
				\int_{\D} \left[W(h_+ e^{u_{f_j}}/\eps)-W(h_+/\eps) \right]\, \sigma_{\Id}^+\, \dd x\, \dd y & = \int_0^1 \int_{\s^1} \left[W(h_+ e^{u_{f_j}}/\eps)-W(h_+/\eps) \right]\, \sigma_{\Id}^+\, r\, \dd\theta\, \dd r\\[2mm]
				& = \int_{\s^1}  \int_0^{\infty}\Phi_a(\Pcal_{r_{T,a}} u_{j,+})\, \dd t\, \dd\theta \\[2mm]
				& = \int_{\s^1} \int_{-T}^{\infty} \Phi_a(\Pcal_{r_{T,a}} u_{j,+})\, \dd a\, \dd\theta.
			\end{aligned} 
		\end{equation}
		For the remaining part, we consider the identity $\dr_r \Pcal_r u = r^{-1}\, \Lambda \Pcal_r u$, which follows from the Fourier-series expansion \eqref{eq:Poisson-Fourier-definition}. Hence, we have 
		\begin{equation*}
			\frac{2r}{1+r^2}\, \dr_r\Pcal_r u = \frac{2}{1+r^2}\, \Lambda \Pcal_r u.
		\end{equation*}
		By Lemma \ref{lm:decompo-Lf}, we have 
		\begin{equation*}
			\begin{aligned}
				\int_{\D} W(h_+ e^{u_{f_j}}/\eps)\, L^+_{f_j} & = \int_{\s^1} \int_0^1 W(h_+ e^{u_{f_j}}/\eps)\, L^+_{f_j}\, r\, \dd r\, \dd \theta \\[2mm]
				& = \int_{\s^1} \int_0^{\infty} W(h_+ e^{u_{f_j}}/\eps)\, \frac{2r}{1+r^2}\, \dr_r \Pcal_r u_{j,+}\,\dd t\, \dd \theta \\[2mm]
				& = \int_{\s^1} \int_{-T}^{\infty} W(h_+ e^{u_{f_j}}/\eps)\, \frac{2}{1+r_{T,a}^2}\, \Lambda \Pcal_r u_{j,+}\,\dd a\, \dd \theta.
			\end{aligned}
		\end{equation*}
		Since $\int_{\s^1}\Lambda \Pcal_r u_{j,+}=0$, we have 
		\begin{equation*}
			\begin{aligned}
				\int_{\D} W(h_+ e^{u_{f_j}}/\eps)\, L^+_{f_j}\, \dd x\, \dd y & = \int_{\s^1} \int_{-T}^{\infty} \left[ W(h_+ e^{u_{f_j}}/\eps) - W(h_+/\eps) \right] \, \frac{2}{1+r_{T,a}^2}\, \Lambda \Pcal_r u_{j,+}\,\dd a\, \dd \theta.
			\end{aligned}
		\end{equation*}
		Using \eqref{eq:W-Phi} in the preceding relation together with \eqref{eq:sigma-term-B0+}, we obtain the following expression for $B_{\eps,+}^0$:
		\begin{equation}\label{eq:interior-round-linear-Poisson}
			B_{\eps,+}^0(f_j) = \int_{-T}^{\infty}\int_{\s^1} \Phi_a(\Pcal_{r_{T,a}}u_{j,+})\,\dd\theta\, \dd a + \int_{-T}^{\infty} \frac{2}{1+r_{T,a}^2} \Er\left( \Pcal_{r_{T,a}}u_{j,+}, \Phi_a(\Pcal_{r_{T,a}}u_{j,+}) \right)\dd a.
		\end{equation}
		We proceed similarly for $B_{\eps,-}$. Using \eqref{eq:dt-sigma-}, we first obtain 
		\begin{equation*}
			\int_{\D^*} \left[W(h_- e^{u_{g_j}}/\eps)-W(h_-/\eps) \right]\, \sigma_{\Id}^-\, \dd x\, \dd y  = - \int_{\s^1} \int_{-T}^{\infty} \Phi_a(\Pcal_{r_{T,a}} u_{j,-})\, \dd a\, \dd\theta.
		\end{equation*}
		Then, using Lemma \ref{lm:decompo-Lg} and $u_j(s e^{\ii\theta}) = \Pcal_{\frac{1}{s}} u_{j,-}(\theta)$, we obtain
		\begin{equation}\label{eq:exterior-round-linear-Poisson}
			B_{\eps,-}^0(g_j) = -\int_{-T}^{\infty}\int_{\s^1} \Phi_a(\Pcal_{r_{T,a}}u_{j,-})\,\dd a\, \dd\theta + \int_{-T}^{\infty} \frac{2r_{T,a}^2}{1+r_{T,a}^2}\, \Er\left( \Pcal_{r_{T,a}}u_{j,-}, \Phi_a(\Pcal_{r_{T,a}}u_{j,-}) \right)\dd a.
		\end{equation}
		Proposition \ref{prop:uniform-Poisson-regularization}, together with Lemma \ref{lem:uniform-height-correction}, implies
		\eqref{eq:uniform-cutoff-limit}.
	\end{proof}
	
	We set
	\begin{equation}\label{eq:B-limit}
		B(f,g)= \frac{1}{2}\int_{\s^1}(u_+-u_-)\,\dd\theta +\frac{1}{2}\left( \Er(u_+,u_+) +\Er(u_-,u_-)\right).
	\end{equation}
	
	\begin{corollary}
		For every analytic Jordan curve $\Gamma$, we have
		\begin{equation}\label{eq:analytic-finite-part-volume}
			V(\Gamma) = B(f,g) +\frac{1}{2}\int_{\D} R^+_f\,\dd x\, \dd y +\frac{1}{2}\int_{\D^*}R_g^-\,\dd x\, \dd y.
		\end{equation}
	\end{corollary}
	
	\begin{proof}
		Since $\Gamma$ is analytic, we do not need an approximation and may work directly with the constant sequence $\Gamma_j=\Gamma$.
		Then \eqref{eq:uniform-cutoff-limit} implies that $B_\eps(f,g)\to B(f,g)$ as $\eps\to 0$. By Lemma \ref{lm:decompo-sigmaf} and Lemma \ref{prop:exterior-epstein-flux}, the remaining terms in the expression for $V(\Gamma)$ are
		\begin{equation*}
			\int_\D W(\xi_f/\eps)R_f^+\,\dd x\,\dd y +\int_{\D^*}W(\xi_g/\eps)R_g^-\,\dd x\, \dd y.
		\end{equation*}
		We have $W(\xi/\eps)\to\frac{1}{2}$ a.e.\ and $0\leq W\leq \frac{1}{2}$. The integrability of $R_f^+$ and $R_g^-$ in \eqref{eq:global-remainders-L1} provides an $\Ll^1$ bound. We obtain \eqref{eq:analytic-finite-part-volume} by dominated convergence.
	\end{proof}
	
	\section{Continuity of $\VR$}\label{sec:continuity}
	
	\begin{proof}[Proof of \Cref{th:continuity}]
		The second term in the definition of $\VR$ in \eqref{eq:intro-renormalized-volume} is continuous by \cite[Theorem 1.4]{BBVPW2025}. Therefore, we focus on the quantity $V(\Gamma)\coloneqq \lim_{\eps\to 0^+} V_{\eps}(\Gamma)$ defined in \eqref{eq:volume-boundary-primitive}.
		
		Let $(\Gamma_j)_{j\in\N}$ be as in \eqref{eq:approximation}.  
		By \eqref{eq:normalisation}, \eqref{eq:WP-convergence-data}, and \eqref{eq:trace-energies}, we have strong convergence $u_{j,+}\to u_+$ and $u_{j,-}\to u_-$ in $\W^{\frac{1}{2},2}(\s^1)$.
		
		We apply \eqref{eq:exact-weighted-volume-flux} to every analytic $\Gamma_j$, using the decomposition of $\sigma_f$ in Lemma \ref{lm:decompo-sigmaf}, that of $\sigma_g$ in Lemma \ref{prop:exterior-epstein-flux}, the cancellation \eqref{eq:round-integral-cancellation}, and the notation \eqref{eq:Beps}:
		\begin{equation*}
			V_\eps(\Gamma_j) = B_{\eps}(f_j,g_j) +\int_\D W(\xi_{f_j}/\eps)\, R_{f_j}^+\,\dd^2 z  + \int_{\D^*} W(\xi_{g_j}/\eps)\, R_{g_j}^-\,\dd^2 z.
		\end{equation*}
		By \eqref{eq:analytic-finite-part-volume}, we have 
		\begin{align*}
			V(\Gamma_j) = B(f_j,g_j) +\frac12\int_\D R_{f_j}^+\,\dd^2 z +\frac12\int_{\D^*}R_{g_j}^-\,\dd^2 z.
		\end{align*}
		Hence,
		\begin{equation}\label{eq:exact-cutoff-difference}
			V(\Gamma_j)-V_{\eps}(\Gamma_j) = B(f_j,g_j)-B_{\eps}(f_j,g_j) +\int_\D \left(\frac{1}{2}- W (\xi_{f_j}/\eps) \right)\, R_{f_j}^+\,\dd^2 z +\int_{\D^*} \left(\frac{1}{2}-W(\xi_{g_j}/\eps) \right)\, R_{g_j}^-\,\dd^2 z.
		\end{equation}
		By Lemma \ref{thm:uniform-cutoff-limit}, we have
		\begin{equation*}
			\lim_{\eps\to 0} \sup_{j\in\N} |B(f_j,g_j)-B_{\eps}(f_j,g_j)| =0.
		\end{equation*}
		From \eqref{eq:large-J} and \eqref{eq:cutoff-profile}, we have $W(\xi/\eps)=\frac{1}{2}$ for $\xi\geq 2\eps$. 
		Let $\delta\in(0,1)$. By Lemma \ref{lem:small-height-localization}, for $\eps>0$ small enough, we have 
		\begin{equation*}
			\begin{aligned} 
				& \int_\D \left(\frac{1}{2}- W (\xi_{f_j}/\eps) \right)\, R_{f_j}^+\,\dd^2 z +\int_{\D^*} \left(\frac{1}{2}-W(\xi_{g_j}/\eps) \right)\, R_{g_j}^-\,\dd^2 z \\[2mm]
				& = \int_{\D\setminus \D_{1-\delta}} \left(\frac{1}{2}- W (\xi_{f_j}/\eps) \right)\, R_{f_j}^+\,\dd^2 z +\int_{\D_{1+\delta}\setminus \D} \left(\frac{1}{2}-W(\xi_{g_j}/\eps) \right)\, R_{g_j}^-\,\dd^2 z.
			\end{aligned} 
		\end{equation*}
		The pointwise estimates of $R_f^+$ in Lemma \ref{lm:decompo-sigmaf} and $R_g^-$ in Lemma \ref{prop:exterior-epstein-flux} imply that 
		\begin{equation*}
			\begin{aligned} 
				& \left| \int_\D \left(\frac{1}{2}- W (\xi_{f_j}/\eps) \right)\, R_{f_j}^+\,\dd^2 z +\int_{\D^*} \left(\frac{1}{2}-W(\xi_{g_j}/\eps) \right)\, R_{g_j}^-\,\dd^2 z \right| \\[2mm]
				& \leq  \int_{\D\setminus \D_{1-\delta}} \left(|\Nc(f_j)(z)|^2+ (1-|z|^2)^2|\Sc(f_j)(z)|^2 \right)\,\dd^2 z \\[2mm]
				&\qquad  +\int_{\D_{1+\delta}\setminus \D} \left(|\Nc(g_j)(z)|^2 + (1-|z|^2)^2|\Sc(g_j)(z)|^2 \right)\,\dd^2 z .
			\end{aligned} 
		\end{equation*}
		By \eqref{eq:WP-convergence-data} and \eqref{eq:WP-convergence-data2}, the integrands on the right-hand side are uniformly integrable in $j$. Therefore, we obtain
		\begin{equation*}
			\limsup_{\eps\to 0^+}\ \sup_{j} \left| \int_\D \left(\frac{1}{2}- W (\xi_{f_j}/\eps) \right)\, R_{f_j}^+\,\dd^2 z +\int_{\D^*} \left(\frac{1}{2}-W(\xi_{g_j}/\eps) \right)\, R_{g_j}^-\,\dd^2 z \right| =0.
		\end{equation*}
		We conclude that 
		\begin{equation*}
			\limsup_{\eps\to 0^+} \sup_j |V(\Gamma_j) - V_{\eps}(\Gamma_j)|=0.
		\end{equation*}
		Therefore,
		\begin{equation*}
			\begin{aligned}
				\limsup_{j\to\infty} |V(\Gamma_j)-V(\Gamma)| & \leq \limsup_{\eps\to 0^+}\limsup_{j\to\infty} \Big[ \left| V(\Gamma_j)-V_{\eps}(\Gamma_j) \right|+\left| V_{\eps}(\Gamma_j) - V_{\eps}(\Gamma) \right|+ \left|V_{\eps}(\Gamma)-V(\Gamma) \right| \Big]  \\[2mm]
				& \leq \limsup_{\eps\to 0^+}\limsup_{j\to\infty} \Big[ \left|  V_{\eps}(\Gamma_j) - V_{\eps}(\Gamma) \right|+ \left| V_{\eps}(\Gamma)-V(\Gamma) \right| \Big] .
			\end{aligned}
		\end{equation*}
		For every $\eps>0$, \cite[Lemma 5.17]{BBVPW2025} also gives $V_{\eps}(\Gamma_j) \to V_{\eps}(\Gamma)$ as $j\to\infty$. Hence, we obtain
		\begin{equation*}
			\limsup_{j\to\infty} |V(\Gamma_j)-V(\Gamma)|\leq \limsup_{\eps\to 0}\left|   V_{\eps}(\Gamma)-V(\Gamma) \right|.
		\end{equation*}
		By definition, $V(\Gamma)=\lim_{\eps\to 0}V_{\eps}(\Gamma)$. Hence, $V(\Gamma_j)\to V(\Gamma)$ as $j\to\infty$.
		
		We now let $(\Gamma_k)_{k\in\N}$ be an arbitrary sequence of WP curves converging to $\Gamma$ in the WP topology. For each $k\in\N$, we consider the approximation sequence $(\Gamma_{k,j})_{j\in\N}$ of \eqref{eq:approximation} for $\Gamma_k$. Then the family $(\Gamma_{k,j})_{k,j\in\N}$ forms a compact family in the WP topology (because of the choice of approximation sequence). Hence, we can apply all the preceding arguments uniformly and conclude the proof of Theorem \ref{th:continuity}.
	\end{proof}

	\section{Comparison between $\Ar$ and $\Sr$}\label{sec:comparison}
	
	\subsection{Grunsky operator and O'Hara--Solanes functional}
	Let $\Gamma\subset\C$ be a \underline{smooth} Jordan curve separating two connected sets $\Omega_{\pm}$ with $0\in\Omega_+$ and $\Omega_+$ bounded.
	The O'Hara--Solanes functional \cite[Definition 3.11]{OS2015} is given by
	\begin{equation}\label{eq:EOS_v0}
		\EOS(\Gamma) \coloneqq \lim_{\eps\to 0} \left( \frac{2\, L(\Gamma)}{\eps} - \int_{(\Omega_+\times \Omega_-)\setminus \{ (z,w) : |z-w|<\eps \}} \frac{\dd z\, \dd w}{|z-w|^4} \right).
	\end{equation}
	We have the following expression: 
	\begin{claim}\label{cl:EOS_v1}
		Let $\nu$ be the outward-pointing normal of $\Omega_+$. We have
		\begin{equation*}
			\EOS(\Gamma) = \frac{1}{2} \int_{\Gamma\times \Gamma}  \frac{ \scal{y-x}{\nu(y)}\, \scal{x-y}{\nu(x)} }{|x-y|^4}\, \dd x\, \dd y.
		\end{equation*}
	\end{claim}
	\begin{proof}
		Given $p,q\in\Gamma$, we denote by $\theta_p(q)$ (resp.\ $\theta_q(p)$) the angle between $q-p$ and $\Gamma$ at $p$ (resp.\ at $q$). Then, by Proposition 3.13 in \cite{OS2015}, the functional $\EOS$ has the following expression:
		\begin{equation*}
			\EOS(\Gamma) = -\frac{1}{2} \int_{\Gamma\times \Gamma}  \frac{ \sin(\theta_p(q))\, \sin(\theta_q(p)) }{|p-q|^2}\, \dd p\, \dd q.
		\end{equation*}
		By the definition of $\theta_p$ and $\theta_q$, we have
		\begin{equation*}
			\scal{\frac{q-p}{|q-p|}}{\nu(q)} = -\sin(\theta_q(p)), \qquad \text{ and }\qquad \scal{\frac{q-p}{|q-p|} }{\nu(p)} = -\sin(\theta_p(q)).
		\end{equation*}
	\end{proof}
	We consider the uniformizing map $f\colon \D\to \Omega_+$ with $f(0)=0$ and $f'(0)=1$. We then consider the expansion (see, for instance, \cite[Section 2]{TakhtajanTeo2006})
	\begin{equation}\label{eq:expansion-logg}
		u(z,w)\coloneqq \log\frac{f(z)-f(w)}{z-w} = -\sum_{m,n\geq 0} b_{m,n}\, z^m\, w^n.
	\end{equation}
	We define the Grunsky operator $B_1=\left(\sqrt{mn}\, b_{m,n}\right)_{m,n\geq 0}$ acting on $\ell^2(\N)$ with Hilbert--Schmidt norm
	\begin{equation*}
		\|B_1\|_{\rm HS}^2 = \sum_{m,n\geq 0} mn\, |b_{m,n}|^2.
	\end{equation*}
	By \cite[Chapter 2, Theorem 2.10]{TakhtajanTeo2006}, if a sequence of WP curves $(\Gamma_k)_{k\in\N}$ converges to some WP curve $\Gamma$ in the WP topology, then the associated Grunsky operators converge in Hilbert--Schmidt norm.
	Using the Grunsky operator, we can rewrite the O'Hara--Solanes energy as follows and thus obtain an expression for WP curves. 
	
	\begin{lemma}\label{lm:EOS-B1}
		Let $\Gamma$ be a smooth Jordan curve with Grunsky operator $B_1$. Then we have
		\begin{equation}\label{eq:EOS-B1}
			\EOS(\Gamma) = \frac{\pi^2}{2} + \pi^2\, \|B_1\|_{\rm HS}^2.
		\end{equation}
	\end{lemma}
	\begin{remark} 
		For a WP curve $\Gamma$, one can define $\EOS(\Gamma)$ by the formula \eqref{eq:EOS-B1}. It is straightforward to check that the computations remain valid, so Claim \ref{cl:EOS_v1} also holds for WP curves. One can then check that the computations in the proofs of Propositions 3.12 and 3.13 also remain valid. Therefore, the formula \eqref{eq:EOS_v0} holds for WP curves.
	\end{remark} 
	\begin{proof}[Proof of Lemma \ref{lm:EOS-B1}]
		The curve $\Gamma$ can be parametrized by the map $\gamma(\theta) = f(e^{\ii\theta})$ for $\theta\in[0,2\pi]$. The outward normal vector is given by $-\ii \frac{\gamma'(\theta)}{|\gamma'(\theta)|} = e^{\ii\theta}\frac{f'(e^{\ii\theta})}{|f'(e^{\ii\theta})|}$.
		\begin{equation}\label{eq:EOS1}
			\begin{aligned}
				\EOS(\Gamma) & = \frac{1}{2}\int_{\s^1\times \s^1} \frac{ \scal{f(z)-f(w)}{z\, \frac{f'(z)}{|f'(z)|} }\,  \scal{f(w)-f(z)}{w\, \frac{f'(w)}{|f'(w)|}  }}{|f(z) - f(w)|^4}\, |f'(z)|\, |f'(w)|\, |\dd z|\, |\dd w| \\[2mm]
				& = \frac{1}{2}\int_{\s^1\times \s^1} \frac{ \scal{f(z)-f(w)}{z\, f'(z)}\,  \scal{f(w)-f(z)}{w\, f'(w) }}{|f(z) - f(w)|^4}\, |\dd z|\, |\dd w| \\[2mm]
				& = \frac{1}{2}\int_{\s^1\times \s^1} \Re\left( \frac{ z\, f'(z)}{ f(z)-f(w) } \right)\, \Re\left( \frac{ w\, f'(w)}{ f(w)-f(z) } \right)\, |\dd z|\, |\dd w|.
			\end{aligned} 
		\end{equation}
		On the other hand, differentiating \eqref{eq:expansion-logg}, we obtain
		\begin{equation}\label{eq:der-u}
			\begin{aligned}
				\dr_z u(z,w) & = \frac{z-w}{f(z)-f(w)}\, \left( \frac{f'(z)}{z-w} - \frac{f(z) - f(w)}{(z-w)^2} \right) \\[2mm]
				& =\frac{f'(z)}{f(z)-f(w)} - \frac{1}{z-w}.
			\end{aligned}
		\end{equation}
		Hence, we obtain 
		\begin{equation*}
			\begin{aligned}
				\Re\left( \frac{ z\, f'(z)}{ f(z)-f(w) } \right) &= \Re\left(z\, \dr_z u(z,w) \right) + \Re\left(\frac{z}{z-w}\right).
			\end{aligned}
		\end{equation*}
		For $z,w\in\s^1$, we have 
		\begin{equation*}
			\Re\left(\frac{z}{z-w}\right) = \Re\left(\frac{z(\bar{z} - \bar{w})}{|z-w|^2}\right) = \frac{1-\scal{z}{w}}{|z-w|^2} = \frac{1}{2}. 
		\end{equation*}
		Returning to \eqref{eq:EOS1}, we obtain 
		\begin{equation*}
			\begin{aligned} 
				\EOS(\Gamma) & = \frac{1}{2}\int_{\s^1\times \s^1} \left( \frac{1}{2} + \Re(z\, \dr_z u) \right) \left( \frac{1}{2} + \Re(w\, \dr_w u) \right)\, |\dd z|\, |\dd w| \\[2mm]
				& = \frac{1}{2}\int_{\s^1\times \s^1} \left( \frac{1}{4} + \frac{1}{2}\, \Re(z\, \dr_z u) + \frac{1}{2} \, \Re(w\, \dr_w u) + \Re(z\, \dr_z u)\, \Re(w\, \dr_w u) \right)\, |\dd z|\, |\dd w| \\[2mm]
				& = \frac{\pi^2}{2} +\frac{1}{2}\,\int_{\s^1}  \Re(z\dr_z u)\, |\dd z|  + \frac{1}{2}\int_{\s^1\times\s^1} \Re(z\, \dr_z u) \, \Re(w\, \dr_w u)\, |\dd z|\, |\dd w|.
			\end{aligned} 
		\end{equation*}
		Since $\dr_z u$ is holomorphic on $\D$, the second term satisfies
		\begin{equation*}
			\int_{\s^1}  \Re\big(z\dr_z u(z,w)\big)\, |\dd z| = \Re\left(\int_{\s^1} z\, \dr_z u(z,w)\, |\dd z|\right) = \Re\left(\frac{1}{\ii} \int_0^{2\pi} \dr_z u\, \dd z  \right) =0. 
		\end{equation*}
		We obtain 
		\begin{equation*}
			\begin{aligned} 
				\EOS(\Gamma) & = \frac{\pi^2}{2} + \frac{1}{2}\int_{\s^1\times\s^1} \Re(z\, \dr_z u) \, \Re(w\, \dr_w u)\, |\dd z|\, |\dd w| \\[2mm]
				& = \frac{\pi^2}{2} + \frac{1}{2}\int_{\s^1\times\s^1} \Re\big( z\, \dr_z u \, \Re(w\, \dr_w u) \big)\, |\dd z|\, |\dd w| \\[2mm]
				& = \frac{\pi^2}{2} + \frac{1}{4}\int_{\s^1\times\s^1} \Re\left[ (z\, \dr_z u) \, (w\, \dr_w u + \overline{w\, \dr_w u} )\right]\, |\dd z|\, |\dd w|.
			\end{aligned} 
		\end{equation*}
		Since $\dr_z u\, \dr_w u$ is holomorphic in $\D$, the first term in the above integral satisfies
		\begin{equation*}
			\begin{aligned}
				\int_{\s^1} (z\, \dr_z u(z,w))(w\, \dr_w u(z,w))\, |\dd z| & = \frac{w}{\ii}\int_{\s^1} \dr_z u(z,w)\, \dr_w u(z,w)\, \dd z =0.
			\end{aligned}
		\end{equation*}
		Hence, we obtain 
		\begin{equation*}
			\EOS(\Gamma) = \frac{\pi^2}{2} + \frac{1}{4}\, \Re\left( \int_{\s^1\times\s^1} (z\, \dr_z u) (\bar{w}\, \overline{\dr_w u})\, |\dd z|\, |\dd w| \right).
		\end{equation*}
		Using \eqref{eq:expansion-logg}, we obtain 
		\begin{equation*}
			z\, \dr_z u = -\sum_{m,n\geq 0} m\, b_{m,n}\, z^m\, w^n.
		\end{equation*}
		Hence, we have
		\begin{equation*}
			\begin{aligned} 
				\int_{\s^1\times\s^1} (z\, \dr_z u) (\bar{w}\, \overline{\dr_w u})\, |\dd z|\, |\dd w| & = \sum_{m_1,n_1,m_2,n_2\geq 0} m_1\, n_2\, b_{m_1,n_1}\, \overline{b_{m_2,n_2}}\int_{\s^1\times \s^1} z^{m_1-m_2}\, w^{n_1-n_2}\, |\dd z|\, |\dd w| \\[2mm]
				& = 4\pi^2 \sum_{m,n\geq 0} m\, n\, |b_{m,n}|^2\\[2mm]
				& = 4\pi^2\, \|B_1\|_{\rm HS}^2.
			\end{aligned}
		\end{equation*}
	\end{proof}
	
	\subsection{Comparison between $\Ar$ and $\Sr$} In the case of a \underline{smooth} Jordan curve $\Gamma$, the comparison between $\Ar(\Gamma)$ and $\Sr(\Gamma)$ is an application of the results in \cite{OS2015}. Let $Y\subset \Hd^3$ be a minimal surface with $\dr Y=\Gamma$ of class $\Cc^2$ up to the boundary.
	By Equations (23) and (24) in \cite{OS2015} (see also Theorem 1 in \cite{S2010}), we have\footnote{In hyperbolic space, we have $-\frac{1}{2}|\II|^2_{g_Y} = \det_{g_Y}\II =K_Y+1$, where $K_Y$ denotes the intrinsic curvature.}
	\begin{equation*}
		-\frac{1}{2}\int_Y |\II|^2_{g_Y}\, \dd\vol_{g_Y} = \int_Y (K_Y+1)\, \dd\vol_{g_Y} = 2\pi\chi(Y) + \frac{2}{\pi}\int_{\C\times \C} \left(\sharp(\ell_{wz}\cap Y) - \lambda^2(w,z,\Gamma)\right)\, \frac{\dd^2 w\, \dd^2 z}{|z-w|^4} - \frac{4}{\pi}\, \EOS(\Gamma),
	\end{equation*}
	where 
	\begin{itemize}
		\item for $z,w\in\C$, the set $\ell_{wz}$ is the geodesic in $\Hd^3$ with ideal endpoints $w$ and $z$,
		\item the number $\lambda^2(w,z,\Gamma)$ is $1$ if $w$ and $z$ lie in different components of $\C\setminus \Gamma$ and $0$ otherwise.
	\end{itemize}
	In particular, the second term on the right-hand side is nonnegative, and we obtain the inequality
	\begin{equation}\label{eq:upper-Willmore}
		-\frac{1}{2}\int_Y |\II|^2_{g_Y}\, \dd\vol_{g_Y} \geq 2\pi\, \chi(Y) - \frac{4}{\pi}\, \EOS(\Gamma).
	\end{equation}

	We denote by $(s_j)_{j\geq 0}$ the singular values of $B_1$, all lying in $[0,1)$ (see, for instance, \cite[Remark 2.3, Chapter 2]{TakhtajanTeo2006}). Then we have
	\begin{equation*}
		\EOS(\Gamma) -\frac{\pi^2}{2} = \pi^2\sum_{j\geq 0} s_j^2 \leq \pi^2\sum_{j\geq 0}\left[-\log(1-s_j^2)\right] = -\pi^2\, \log\det(I-B_1^*B_1).
	\end{equation*}
	Using Definition 2.15 and Theorem 3.8 in Chapter 2 of \cite{TakhtajanTeo2006}, we obtain 
	\begin{equation*}
		\EOS(\Gamma) \leq \frac{\pi^2}{2} + \frac{\pi}{12}\, \Sr(\Gamma).
	\end{equation*}
	
	In the case of a WP curve $\Gamma$, we now prove the inequality \eqref{eq:upper-Willmore}.
	
	\begin{lemma}\label{lem:wp-solanes}
		Let $\Gamma\subset\C$ be a WP Jordan curve, and let $Y\subset\Hd^3$ be a connected, complete, properly embedded minimal surface with boundary $\Gamma$. Then we have
		\begin{equation*}
		\Ar(Y)\geq-\frac{4}{\pi}\, \EOS(\Gamma).
		\end{equation*}
	\end{lemma}
	\begin{proof}
		We write the hyperbolic metric on $\Hd^3 = \left\{ (x^1,x^2,x^3)\in\R^3 \mid x^3>0 \right\}$ as follows:
		\begin{equation*}
			g_{\Hd^3}= \frac{(\dd x^1)^2+(\dd x^2)^2+(\dd t)^2}{t^2}.
		\end{equation*}
		Given $t>0$, we denote
		\begin{equation}\label{eq:defA_ell}
			A(t) \coloneqq \vol_{\Hd^3}(Y\cap\{x^3>t\}) \qquad \text{ and } \qquad  \ell(t) \coloneqq \mathcal H^1(Y\cap\{x^3=t\}).
		\end{equation}
		By \cite[Corollary~1.7]{Bishop2025}, we have 
		\begin{equation}\label{eq:wp-ra-input}
			A(t)-\frac{\ell(t)}{t} \xrightarrow[t\to 0]{} \Ar(Y).
		\end{equation}
		From the paragraph ``Length convergence on minimal surfaces'' in Section 23 of \cite{Bishop2025}, we have
		\begin{equation}\label{eq:wp-length-input}
			\int_0^1\frac{|\ell(t)-L|}{t^2}\,\mathrm dt<\infty,
		\end{equation}
		where $L \coloneqq \ell(0)$ is the length of $\Gamma$.
		
		Given $h>0$, we consider the Euclidean vertical translation of $Y$ defined by
		\begin{equation*}
			Z_h\coloneqq \left\{(x^1,x^2,x^3+h) \mid (x^1,x^2,x^3)\in Y \right\}.
		\end{equation*}
		Its boundary is $\Gamma\times\{h\}$. We first prove the following limit.
		\begin{claim} 
			We have
		\begin{equation}\label{eq:wp-translation-limit}
			\vol_{\Hd^3}(Z_h)-\frac{L}{h} \xrightarrow[h\to 0]{} \Ar(Y).
		\end{equation}
		\end{claim} 
		\begin{proof} 
		Let $g_Y^e$ be the first fundamental form of $Y$ induced by the Euclidean metric in $\R^3$ and $g^h_Y = (x^3)^{-2}g_Y^e$ the first fundamental form induced by the hyperbolic metric in $\Hd^3$. The volume of $Z_h$ is given by
		\begin{equation*}
			\vol_{\Hd^3}(Z_h) =\int_Y\frac{\dd\vol_{g_Y^e}(x)}{(x^3+h)^2} =\int_Y\left(\frac{x^3}{x^3+h}\right)^2\dd\vol_{g_Y^h}(x).
		\end{equation*}
		We have
		\begin{equation*}
			\frac{\dd}{\dd t}\left(\frac{t}{t+h}\right)^2 = \frac{2h t}{(t+h)^3} \eqqcolon w_h(t).
		\end{equation*}
		With the notation in \eqref{eq:defA_ell}, we obtain
		\begin{equation}\label{eq:wp-translation-average}
			\vol_{\Hd^3}(Z_h) = \int_Y \left(\int_0^{\infty} \mathbf{1}_{\{x^3\geq t\}}\, w_h(t)\, \dd t\right)\, \dd\vol_{g_Y^h}(x) =  \int_0^\infty w_h(t)\, A(t)\,\dd t.
		\end{equation}
		By definition, $w_h(t)$ satisfies the following three properties:
		\begin{itemize}
			\item The total integral is $1$: 
			\begin{equation}\label{eq:int-wh}
				\int_0^\infty w_h(t)\,\dd t= \left[\left( \frac{t}{t+h} \right)^2 \right]_{t=0}^{t\to \infty} = 1.
			\end{equation}
			
			\item We have
			\begin{equation}\label{eq:int-wht}
				\int_0^{\infty} \frac{ w_h(t) }{t}\dd t = \int_0^{\infty} \frac{2h}{(h+t)^3}\, \dd t = \frac{2}{h}\int_0^{\infty} \frac{\dd s}{(1+s)^3} = \frac{1}{h}.
			\end{equation}
			
			\item For any $a>0$, we have the limit
			\begin{equation}\label{eq:tail-wh1}
				\lim_{h\to 0}\int_0^a w_h(t)\,\dd t= \lim_{h\to 0} \left(\frac{a}{a+h}\right)^2 =1.
			\end{equation}
			In particular, we obtain 
			\begin{equation}\label{eq:tail-wh2}
				\lim_{h\to 0}\int_a^{\infty} w_h(t)\,\dd t=0.
			\end{equation}
		\end{itemize}
		By \eqref{eq:wp-ra-input} and \eqref{eq:tail-wh1}-\eqref{eq:tail-wh2}, with the notation in \eqref{eq:defA_ell}, we have 
		\begin{equation*}
			\begin{aligned} 
			& \limsup_{h\to 0} \left|\int_0^1 w_h(t) \left(A(t)-\frac{\ell(t)}{t} - \Ar(Y)\right) \,\dd t\right| \\[2mm]
			& \leq \limsup_{a\to 0} \limsup_{h\to 0} \left( \int_0^a w_h(t) \left|A(t)-\frac{\ell(t)}{t} - \Ar(Y)\right| \,\dd t + \int_a^1 w_h(t) \left|A(t)-\frac{\ell(t)}{t} - \Ar(Y)\right| \,\dd t \right) \\[2mm]
			& \leq \limsup_{a\to 0} \limsup_{h\to 0} \left( \left\|A(t)-\frac{\ell(t)}{t} - \Ar(Y)\right\|_{\Ll^{\infty}(0,a)} \int_0^a w_h(t) \,\dd t + \left\|A(t)-\frac{\ell(t)}{t} - \Ar(Y)\right\|_{\Ll^{\infty}(0,1)}\int_a^1 w_h(t) \,\dd t \right) \\[2mm]
			& \leq \limsup_{a\to 0}  \left\|A(t)-\frac{\ell(t)}{t} - \Ar(Y)\right\|_{\Ll^{\infty}(0,a)} =0.
			\end{aligned} 
		\end{equation*}
		Thus, we obtain
		\begin{equation*}
			\int_0^1 w_h(t) \left(A(t)-\frac{\ell(t)}t\right)\,\dd t \xrightarrow[h\to 0]{}\Ar(Y).
		\end{equation*}
		On the other hand, we deduce from dominated convergence and \eqref{eq:wp-length-input} that
		\begin{equation*}
			\int_0^1 w_h(t)\, \frac{\ell(t)-L}{t}\,\dd t
			=\int_0^1 \frac{2ht^2}{(t+h)^3}\, \frac{\ell(t)-L}{t^2}\,\dd t \xrightarrow[h\to 0]{} 0.
		\end{equation*}
		On $[1,\infty)$, the function $\left(t\mapsto A(t)-L/t\right)$ is bounded. Hence, we obtain from \eqref{eq:tail-wh2} that
		\begin{equation*}
			\int_1^{\infty} w_h(t)\,\left( A(t) - \frac{L}{t}\right)\,\dd t \xrightarrow[h\to 0]{} 0.
		\end{equation*}
		Therefore, we obtain the following limit:
		\begin{equation*}
			\vol_{\Hd^3}(Z_h)-\frac{L}{h} =\int_0^\infty w_h(t) \left(A(t)-\frac{L}{t}\right)\,\dd t = \int_0^\infty w_h(t) \left(A(t) - \frac{\ell(t)}{t} +\frac{\ell(t)-L}{t}\right)\,\dd t \xrightarrow[h\to 0]{}\Ar(Y).\qedhere
		\end{equation*}
		\end{proof} 
		
		For distinct $z,w\in\C$, we denote by $N_{Z_h}(z,w)$ the number of intersections of $Z_h$ with the hyperbolic geodesic in $\Hd^3$ having ideal endpoints $z,w$.
		Combining Equations (2) and (16) in \cite{S2010}, we obtain\footnote{With the notation of \cite{S2010}, it seems a priori that the constant should be $\frac{1}{\pi}$ as in Equation (2). However, the integral in Equation (2) is on the space of unoriented geodesics. It follows from Equation (16) that $\dd\ell = 4\frac{\dd^2 z\, \dd^2 w}{|z-w|^4}$ on the space of oriented geodesics. Our constant $\frac{2}{\pi}$ follows from the fact that the space of oriented geodesics is a double cover of the set of unoriented geodesics. }
		\begin{equation}\label{eq:wp-translation-crofton}
			\vol_{\Hd^3}(Z_h) =\frac{2}{\pi}\int_{\C\times\C}
			\frac{N_{Z_h}(z,w)}{|z-w|^4}\, \dd^2z\,\dd^2w.
		\end{equation}
		We now use a change of variables from the proof of Proposition 3.12 in \cite{OS2015}, to which we refer for details.
		Since a geodesic in $\Hd^3$ is a semicircle in $\R^3$ intersecting $\C$ orthogonally, a geodesic $\ell$ with distinct endpoints $z,w\in\C$ and crossing the plane $\{x^3=h\}$ is uniquely determined by its two intersection points $(x,h)$ and $(y,h)$, with $x\neq y$. Indeed, $\ell$ is the semicircle with center $\frac{x+y}{2}\in\C$ and radius
		\begin{equation*}
			R = \frac{|z-w|}{2} = \sqrt{h^2 + \frac{|y-x|^2}{4}}.
		\end{equation*}
		Therefore, the change of variables from $(z,w)$ to $(x,y)$ gives
		\begin{equation}\label{eq:wp-endpoint-change}
			|z-w|^2=|x-y|^2+4h^2.
		\end{equation}
		Moreover, the change of variables $(x,y)\mapsto \left(\frac{x+y}{2},y-x\right)$ has Jacobian $1$, as does $(z,w)\mapsto\left(\frac{z+w}{2},w-z\right)$. Thus, we have
		\begin{equation}\label{eq:wp-endpoint-change2}
			\dd^2 z\, \dd^2 w = \dd^2 x\, \dd^2 y.
		\end{equation}
		If $x\in\Omega_+$ and $y\in\Omega_-$, the geodesic in $\Hd^3$ passing through $(x,h)$ and $(y,h)$ must intersect $Z_h$.
		Indeed, this geodesic is a continuous path joining one point outside $\overline{Z_h}\cup \left( \overline\Omega_+\times\{h\} \right) $ and one point inside since the intersection at $(x,h)\in \Omega_+\times\{h\}$ is transverse.
		Consequently, it follows from \eqref{eq:wp-translation-crofton},
		\eqref{eq:wp-endpoint-change}, and \eqref{eq:wp-endpoint-change2} that
		\begin{equation}\label{eq:wp-fixed-domain-bound}
			\begin{aligned} 
				\vol_{\Hd^3}(Z_h) & \geq  \frac{2}{\pi}\int_{ (\Omega_+\times \Omega_-)\cup (\Omega_-\times \Omega_+)}
				\frac{\dd^2 x\,\dd^2 y}{(|x-y|^2+4h^2)^2}. \\[2mm]
				& \geq \frac{4}{\pi}  \int_{\Omega_+\times\Omega_-} \frac{\dd^2x\,\dd^2y}{(|x-y|^2+4h^2)^2}.
			\end{aligned} 
		\end{equation}
		We now estimate the right-hand side by following the approach of \cite[Proposition 3.12]{OS2015}. Given $x,y\in\C$, we denote by $N_\Gamma(x,y)$ the number of intersections of the open segment $(x,y)$ with $\Gamma$. We define 
		\begin{equation*} 
			D_\Gamma(x,y) \coloneqq  N_\Gamma(x,y) -\bigl(\mathbf1_{\Omega_+}(x)-\mathbf1_{\Omega_+}(y)\bigr)^2 \geq 0.
		\end{equation*} 
		
		Given $\theta\in[0,2\pi]$, we denote $e_{\theta}\coloneqq \cos(\theta)+\ii\sin(\theta)\in\C$. Let $\gamma\colon [0,L]\to \C$ be an arclength parametrization of $\Gamma$. Given $\theta\in[0,2\pi]$ and $r>0$, we denote $\Phi_{\theta,r}(s,a)\coloneqq \gamma(s)-a\, e_{\theta}$ for $(s,a)\in[0,L]\times(0,r)$. Given $y\in\C$, the preimage $\Phi_{\theta,r}^{-1}(\{y\})$ corresponds to the intersections of $\Gamma$ with the segment $[y+r\, e_{\theta},y]$.
		
		We first make the change of variables $x=y+re_{\theta}$ and then use the area formula:
		\begin{equation}\label{eq:integration_N}
			\begin{aligned} 
				\int_{\{|x-y|>\eps\}} \frac{ N_\Gamma(x,y)}{|x-y|^4}\, \dd^2 x\,\dd^2 y & = \int_{\eps}^{+\infty} \int_0^{2\pi} \left(\int_{\C} N_{\Gamma}(y+r\, e_{\theta},y)\, \dd^2 y\right) \, \dd\theta\, \frac{\dd r}{r^3} \\[2mm]
				& = \int_{\eps}^{+\infty} \int_0^{2\pi}\left( \int_{[0,L]\times(0,r)} \left|\det \dd\Phi_{\theta,r} \right|\, \dd s\, \dd a\right) \, \dd\theta\, \frac{\dd r}{r^3} \\[2mm]
				& = \int_{\eps}^{+\infty} \int_0^{2\pi}\left( \int_0^L\int_0^r \left|\det (\gamma'(s),e_{\theta}) \right|\, \dd a\, \dd s\right) \, \dd\theta\, \frac{\dd r}{r^3} \\[2mm]
				& = \int_{\eps}^{+\infty} \int_0^{2\pi}\left( \int_0^L \left|\det (\gamma'(s),e_{\theta}) \right|\, \dd s\right) \, \dd\theta\, \frac{\dd r}{r^2} \\[2mm]
				& = \int_{\eps}^{+\infty} \int_0^L\left( \int_0^{2\pi} \left|\sin\theta \right|\, \dd \theta\right) \, \dd s\, \frac{\dd r}{r^2} \\[2mm]
				& = 4L\int_{\eps}^{+\infty}\frac{\dd r}{r^2} = \frac{4L}{\eps}.
			\end{aligned} 
		\end{equation}
		Hence, we obtain 
		\begin{equation*}
			\frac{1}{2}\int_{\{|x-y|>\eps\}} \frac{D_\Gamma(x,y)}{|x-y|^4}\,
			\dd^2x\,\dd^2y =
			\frac{2L}{\eps}
			-\int_{\{x\in\Omega_+,\ y\in\Omega_- :  |x-y|>\eps\} }
			\frac{\dd^2x\,\dd^2y}{|x-y|^4}.
		\end{equation*}
		Monotone convergence therefore proves the existence of the limit $\eps\to 0$, and we obtain
		\begin{equation}\label{eq:wp-os-positive}
			\EOS(\Gamma) =\frac{1}{2}\int_{\C\times\C} \frac{D_\Gamma(x,y)}{|x-y|^4}\, \dd^2x\,\dd^2 y.
		\end{equation}
		We now return to \eqref{eq:wp-fixed-domain-bound} by performing a computation similar to \eqref{eq:integration_N}: 
		\begin{equation}\label{eq:lower-EOS}
			\begin{aligned}
				\EOS(\Gamma) & \geq \frac{1}{2}\int_{\C\times\C}
				\frac{D_\Gamma(x,y)}{(|x-y|^2+4h^2)^2}\,
				\dd^2x\,\dd^2y\\[2mm]
				& \geq \frac{1}{2} \int_{\C\times\C}\frac{ N_{\Gamma}(x,y)}{(|x-y|^2+4h^2)^2}\,\dd^2 x\, \dd^2 y
				-\int_{\Omega_+\times\Omega_-} \frac{\dd^2x\,\dd^2y}{(|x-y|^2+4h^2)^2} \\[2mm]
				& \geq 2L \int_0^\infty\frac{ r^2}{(r^2+4h^2)^2}\,\dd r 
				-\int_{\Omega_+\times\Omega_-} \frac{\dd^2x\,\dd^2y}{(|x-y|^2+4h^2)^2} \\[2mm]
				& \geq \frac{\pi L}{4h}-\int_{\Omega_+\times\Omega_-} \frac{\dd^2x\,\dd^2y}{(|x-y|^2+4h^2)^2}.
			\end{aligned}
		\end{equation}
		Indeed, we compute the last integral using the change of variables $r=2hs$ and $s = \frac{r}{2h}$:
		\begin{equation*}
			\begin{aligned} 
				\int_0^\infty\frac{ r^2}{(r^2+4h^2)^2}\,\dd r & = \int_0^{\infty} \frac{ 4h^2\, s^2 }{(4h^2)^2(1+s^2)^2}\, 2h\,\dd s \\[2mm]
				& = \frac{1}{2h}\int_0^{\infty} \frac{s^2}{(1+s^2)^2}\dd s \\[2mm]
				& = \frac{1}{2h}\left( \left[ -\frac{1}{2}\, \frac{s}{1+s^2} \right]_{s=0}^{s\to\infty} + \frac{1}{2}\int_0^{\infty}\frac{\dd s}{1+s^2} \right)\\[2mm]
				& = \frac{1}{4h}\left[\arctan(s)\right]_{s=0}^{s\to\infty} = \frac{\pi}{8h}.
			\end{aligned} 
		\end{equation*}
		Combining \eqref{eq:wp-fixed-domain-bound} and \eqref{eq:lower-EOS}, we obtain
		\begin{equation*}
			\vol_{\Hd^3}(Z_h)-\frac{L}{h} \geq \frac{4}{\pi}\left[ \int_{\Omega_+\times\Omega_-} \frac{\dd^2x\,\dd^2y}{(|x-y|^2+4h^2)^2} - \frac{\pi L}{4h} \right] \geq  -\frac{4}{\pi} \EOS(\Gamma).
		\end{equation*}
		We conclude the proof by letting $h\to 0$ and using
		\eqref{eq:wp-translation-limit}.
	\end{proof}

\end{document}